\documentclass[10pt,reqno]{amsproc}
\usepackage{graphicx} 
\usepackage{amsmath,amsfonts,amssymb,amsthm}
\usepackage{xcolor}
\usepackage{mathpazo}
\usepackage{hyperref}
\usepackage{url}
\usepackage{appendix}
\usepackage{cleveref}
\usepackage{comment}
\usepackage{booktabs,multirow,multicol}
\usepackage{caption} 
\usepackage{svg}

\usepackage{fullpage}
\hypersetup{colorlinks}
\hypersetup{colorlinks=true, linkcolor=red, citecolor=blue, filecolor=blue, urlcolor=blue}
\newcounter{dummy} \numberwithin{dummy}{section}
\newtheorem{proposition}[dummy]{Proposition}
\newtheorem{lemma}[dummy]{Lemma}
\newtheorem{theorem}[dummy]{Theorem}

\newtheorem{corollary}[dummy]{Corollary}

\theoremstyle{remark}
\newtheorem{remark}[dummy]{Remark}

\newcommand{\what}{\widehat}
\newcommand{\dd}{\mathrm{d}}
\newcommand{\SL}{\mathrm{SL}}
\newcommand{\bH}{\mathbb{H}}
\newcommand{\bR}{\mathbb{R}}
\newcommand{\bZ}{\mathbb{Z}}

\newcommand{\cM}{\mathcal{M}}
\newcommand{\QM}{\mathcal{QM}}
\newcommand{\RQM}{\mathcal{RQM}}

\newcommand{\cL}{\mathcal{L}}

\usepackage{etoolbox}
\BeforeBeginEnvironment{minted}{\vspace{2mm}}

\DeclareUnicodeCharacter{3BC}{$\pi$}
\DeclareUnicodeCharacter{3C0}{$\pi$}

\newenvironment{red}{\relax\color{red}}{\hspace*{.5ex}\relax}
\newenvironment{blue}{\relax\color{blue}}{\hspace*{.5ex}\relax}
\newcommand{\ber}{\begin{red}}
\newcommand{\er}{\end{red}}
\newcommand{\beb}{\begin{blue}}
\newcommand{\eb}{\end{blue}}

\title{Algebraic proof of modular form inequalities for optimal sphere packings}
\author{Seewoo Lee}
\date{}

\begin{document}

\begin{abstract}
We give algebraic proofs of Viazovska and Cohn-Kumar-Miller-Radchenko-Viazovska's modular form inequalities for 8 and 24-dimensional optimal sphere packings.
\end{abstract}

\maketitle

\section{Introduction}
\label{sec:intro}

In their celebrated papers \cite{viazovska2017sphere,cohn2017sphere}, Viazovska and Cohn--Kumar--Miller--Radchenko--Viazovska proved the optimality of $E_8$ and Leech lattice sphere packings in dimensions $8$ and $24$, respectively.
Their proof is based on the Cohn and Elkies' linear programming bound \cite{cohn2003new}, which states that the existence of a function on $\mathbb{R}^{d}$ satisfying certain inequalities implies an upper bound on the sphere packing densities.
The conjectural \emph{magic functions} in dimensions $8$ and $24$ are constructed by Viazovska and Cohn et al. as integral transforms of certain (quasi)modular forms.
When the dimension is $8$, the required inequalities for the function and its Fourier transform reduce to the following inequalities on modular forms.
\begin{theorem}[Viazovska {\cite[Theorem 4]{viazovska2017sphere}}]
    \label{thm:d8ineq}
    Define\footnote{Note that the original inequality is written in terms of $\psi_I(z) = z^2\psi_S(-1/z)$, but we found that the above form is more convenient for us to work with.}
    \begin{align}
        \phi_0 &= 1728 \frac{(E_2 E_4 - E_6)^2}{E_4^3 - E_6^2}, \label{eqn:d8forig} \\
        \psi_S &= -128 \left( \frac{\Theta_3^4 + \Theta_2^4}{\Theta_4^8} + \frac{\Theta_2^4 - \Theta_4^4}{\Theta_3^8}\right), \label{eqn:d8gorig}
    \end{align}
    Then
    \begin{align}
        \phi_0(it) - \frac{36}{\pi^2}\psi_S(it) &> 0, \label{eqn:d8ineq1} \\
        \phi_0(it) + \frac{36}{\pi^2}\psi_S(it) &> 0, \label{eqn:d8ineq2}
    \end{align}
    for all $t > 0$.
\end{theorem}

(See Section \ref{sec:prelim} for the definitions of the terms appearing above.)
Similarly, the case of dimension $24$ reduces to the following three inequalities.
Here we abuse notations so that two $\psi_S$ in \eqref{eqn:d8gorig} and \eqref{eqn:d24gorig} are different.

\begin{theorem}[Cohn et al. {\cite[Appendix A]{cohn2017sphere}}]
    \label{thm:d24ineq}
    Define
    \begin{align}
        \varphi &= -\frac{49 E_2^2 E_4^3 - 25 E_2^2 E_6^2 - 48 E_2 E_4^2 E_6 - 25 E_4^4 + 49 E_4 E_6^2}{\Delta^2}, \label{eqn:d24forig} \\
        \psi_S &= - \frac{\Theta_{2}^{20}(2\Theta_{2}^{8} + 7\Theta_{2}^{4} \Theta_{4}^{4} + 7\Theta_{4}^{8})}{\Delta^2}. \label{eqn:d24gorig}
    \end{align}
    Then for all $t > 0$,
    \begin{align}
        \varphi(it) + \frac{432}{\pi^2} \psi_S(it) &< 0, \label{eqn:d24ineq1}\\
        \varphi(it) - \frac{432}{\pi^2} \psi_S(it) &> 0. \label{eqn:d24ineq2}
    \end{align}
    and for all $t \geq 1$,
    \begin{equation}
        t^{10} \left(\varphi\left(\frac{i}{t}\right) - \frac{432}{\pi^2} \psi_S\left(\frac{i}{t}\right)\right) \geq \frac{725760}{\pi}e^{2\pi t}\left(t - \frac{10}{3 \pi}\right). \label{eqn:d24ineq3}
    \end{equation}
\end{theorem}

In dimension $8$, the inequality \eqref{eqn:d8ineq1} is easier to prove than \eqref{eqn:d8ineq2} since we have $\phi_0(it) > 0$ and $- \psi_S(it) > 0$ separately.
However, \eqref{eqn:d8ineq2} is trickier and also looks unnatural in the sense that we need to compare two modular forms with different weights, and it also includes a bizarre constant $\frac{36}{\pi^2}$ (note that both $\phi_0$ and $\psi_S$ have rational Fourier coefficients).
The third inequality \eqref{eqn:d24ineq3} is much more complicated due to the non-modular terms (polynomial and exponential).
The original proofs in \cite{viazovska2017sphere,cohn2017sphere} use approximations of Fourier coefficients and numerical analysis, and it is natural to ask whether there's a more conceptual and simpler proof of the inequalities.
Recently, Romik proposed an alternative proof of \eqref{eqn:d8ineq1} and \eqref{eqn:d8ineq2} \cite{romik2023viazovska}, based on various identities among modular forms and special values of them at $z = i$ (which can be expressed in terms of $\pi$ and $\Gamma(1/4)$).

In this article, we provide simple and \emph{algebraic} proofs of theorems \ref{thm:d8ineq} and \ref{thm:d24ineq} which do not require any approximations or numerical analysis.
For \eqref{eqn:d8ineq2} and \eqref{eqn:d24ineq2}, the idea is simple: the \emph{ratio} of two forms with the input $z = it$ is a \emph{monotone function} in $t$, and the constants $\frac{36}{\pi^2}$ and $\frac{432}{\pi^2}$ naturally arise as \emph{limits} of the quotients as $t \to 0^+$ (see Proposition \ref{prop:d8decrease}, \ref{prop:d8limit}, \ref{prop:d24decrease}, and \ref{prop:d24limit}). 
Figure \ref{fig:d8graph} shows this phenomenon for $d = 8$, where $F$ and $G$ are certain normalizations of $\phi_0$ and $\psi_S$ respectively (see \eqref{eqn:d8f} and \eqref{eqn:d8g}).
We can compute the limit easily using the modular transformation laws of $F$ and $G$, and they coincide with the above constants.
Also, the monotonicity of the quotient is equivalent to
\[
    F'(it) G(it) - F(it) G'(it) > 0,
\]
(where the derivatives are defined as in \eqref{eqn:modformdiff}), which is now a \emph{homogeneous} inequality.
For \eqref{eqn:d24ineq3}, we first reduce it to a simpler inequality by replacing the exponential term with $\Delta$ (Lemma \ref{lem:discineq}).
After rewriting the inequality as \eqref{eqn:d24ineq3new2}, we observe that the resulting function is also monotone (Figure \ref{fig:d24harder_mono} and Proposition \ref{prop:d24harder}).

To prove these inequalities, we develop a general theory of quasimodular forms that are nonnegative on the imaginary axis or have nonnegative Fourier coefficients (Section \ref{sec:posqmf}).
In particular, we study how the positivity behaves under (anti-)(Serre-)derivatives, which are simple but surprisingly useful.
Combined with the differential equations that the modular forms satisfy (\eqref{eqn:d8ssf}, \eqref{eqn:d8ssg}, \eqref{eqn:d24ssf}, and \eqref{eqn:d24ssg}), our theory gives short proofs of the inequalities, by considering their Serre derivatives of appropriate weight (Corollary \ref{cor:serrepos}, Proposition \ref{prop:d8decrease}, Lemma \ref{lem:serrederF16}, Proposition \ref{prop:d24decrease}, and Proposition \ref{prop:d24harder}).
Along the way, we found that the modular forms are closely connected to the \emph{extremal quasimodular forms} defined by Kaneko and Koike \cite{kaneko2006extremal}.
The extremal forms are conjectured to have nonnegative Fourier coefficients, and can be defined recursively when the depth is less than $5$ \cite{kaneko2006extremal,grabner2020quasimodular}. As a byproduct, we prove their conjecture in the case of depth $1$ (Corollary \ref{cor:kkd1}).

\begin{figure}[t]
    \centering
    \includesvg[width=0.6\textwidth]{figs/d8}
    \caption{Graph of the quotient $F(it)/G(it)$ as a function in $t > 0$. Here $F = \Delta \phi_0$ and $G = -2\Delta \psi_S$.}
    \label{fig:d8graph}
\end{figure}

Our proofs are based on several nontrivial identities among modular forms.
These can be checked directly with SageMath \cite{sagemath}, since the rings of quasimodular forms are polynomial rings over certain generators, and we only need to check identities for the corresponding polynomials.
Also, we can easily factor a given modular form or write it as a (nonnegative) linear combination of other modular forms using SageMath,
for example \eqref{eqn:kkd1eq1}, \eqref{eqn:kkd1eq2}, \eqref{eqn:kkd1eq3}, \eqref{eqn:d8ssf}, \eqref{eqn:d8ssg}, \eqref{eqn:d24ssf}, and \eqref{eqn:d24ssg}.
Code is available in the GitHub repository \url{https://github.com/seewoo5/posqmf}, checking all the identities appearing in the paper (see Appendix \ref{subsec:appendix_sage} for more details).
The identities can also be checked by hand, where ``pen-and-paper proofs'' of the identities are given in Appendix \ref{subsec:penpaperproof}.

\subsection*{Acknowledgement}

We thank YoungJu Choie, Henry Cohn, Dan Romik, and Sug Woo Shin for the valuable comments and discussions.
Especially, Romik's talk on his work \cite{romik2023viazovska} at UC Berkeley RTG seminar motivated this project.
We also thank the reviewers for carefully reading this paper
and for providing detailed revision suggestions, which have been incorporated herein.

\section{Preliminaries}
\label{sec:prelim}

\subsection{Quasimodular forms}

For any function $f: \mathbb{H} \to \mathbb{C}$ defined on the complex upper half plane $\mathbb{H} = \{z \in \mathbb{C}: \Im z > 0\}$ and integer $k$, we define the action of $\gamma \in \mathrm{SL}_2(\mathbb{Z})$ as
\[
    (f|_{k} \gamma)(z) := (cz + d)^{-k}f \left(\frac{az + b}{cz + d}\right), \quad \gamma = \begin{bmatrix} a & b \\ c & d \end{bmatrix} \in \mathrm{SL}_{2}(\mathbb{Z}).
\]
We denote $T = \left[\begin{smallmatrix} 1 & 1 \\ 0 & 1\end{smallmatrix}\right]$ and $S = \left[\begin{smallmatrix} 0 & -1 \\ 1 & 0\end{smallmatrix}\right]$ that generate $\mathrm{SL}_{2}(\mathbb{Z})$.
Let $q = e^{2\pi i z}$ for $z \in \mathbb{H}$.
Define the Eisenstein series of weight $2, 4, 6$ as
\begin{align*}
    E_2 &= 1 - 24 \sum_{n\geq 1}\sigma_1(n)q^n, \\
    E_4 &= 1 + 240 \sum_{n \geq 1} \sigma_3(n)q^n, \\
    E_6 &= 1 - 504 \sum_{n \geq 1} \sigma_5(n)q^n,
\end{align*}
where $\sigma_{a}(n) := \sum_{d | n} d^a$.
The last two Eisenstein series $E_4, E_6$ are genuine modular forms.
However, $E_2$ is not a modular form - it is a \emph{quasi}modular form of weight 2 and level 1.
They obey the following transformation laws:
\begin{align}
    (E_{2}|_{2}S)(z) = z^{-2}E_2\left(-\frac{1}{z} \right) &= E_2(z) - \frac{6 i}{\pi z}, \label{eqn:e2trans}\\
    (E_{4}|_{4}S)(z) = z^{-4}E_4\left(-\frac{1}{z} \right) &= E_4(z), \label{eqn:e4trans} \\
    (E_{6}|_{6}S)(z) = z^{-6}E_6\left(-\frac{1}{z} \right) &= E_6(z). \label{eqn:e6trans}
\end{align}
The (graded) ring of quasimodular forms is isomorphic to a polynomial ring in 3 variables with generators $E_2$, $E_4$, and $E_6$ \cite[\S 5.1]{bruinier2008elliptic}, and we define the \emph{depth} of a quasimodular form as the highest degree of $E_2$ in its expression as a polynomial in $E_2$, $E_4$, and $E_6$.
It is closed under the differentiation
\begin{equation}
    \label{eqn:modformdiff}
    F' = DF := \frac{1}{2\pi i } \frac{\dd F}{\dd z} = q \frac{\dd F}{\dd q}, \quad \sum_{n} a_n q^n \mapsto \sum_{n} na_n q^n
\end{equation}
which increases depth by $1$ and weight by $2$.
For the Eisenstein series, we have Ramanujan's identities \cite[Proposition 15, p. 49]{bruinier2008elliptic}
\begin{align}
    E_2' &= \frac{E_2^2 - E_4}{12}, \label{eqn:e2der} \\
    E_4' &= \frac{E_2E_4 - E_6}{3}, \label{eqn:e4der} \\
    E_6' &= \frac{E_2 E_6 - E_4^2}{2}. \label{eqn:e6der}
\end{align}
We write $\QM_{w}^{s} = \QM_{w}^{s}(\SL_{2}(\mathbb{Z}))$  for the space of quasimodular forms of weight $w$ and depth $\leq s$, and $\cM_{w} := \QM_{w}^{0}$ for the space of genuine modular forms of weight $w$.
We denote $\Delta := (E_4^3 - E_6^2) / 1728$ for the discriminant form, which is the unique normalized cusp form of weight 12 and level $\SL_2(\bZ)$.
It can be expressed as an infinite product
\[
\Delta(z) = q \prod_{n \geq 1} (1 - q^n)^{24} = \eta(z)^{24}.
\]

\subsection{Jacobi's thetanull functions}

Jacobi's \emph{thetanull} functions are defined as
\begin{align*}
    \Theta_2(z) &= \theta_{10}(z) = \sum_{n \in \mathbb{Z}} q^{\frac{1}{2}(n + \frac{1}{2})^2} \\
    \Theta_3(z) &= \theta_{00}(z) = \sum_{n \in \mathbb{Z}} q^{\frac{n^2}{2}} \\
    \Theta_4(z) &= \theta_{01}(z) = \sum_{n \in \mathbb{Z}} (-1)^{n}q^{ \frac{n^2}{2}}.
\end{align*}
Here we choose $q^{1/2} = e^{\pi i z}$.
These are weight $1/2$ modular forms of level $\Gamma(2) = \{\gamma \in \SL_{2}(\mathbb{Z}): \gamma \equiv \left[\begin{smallmatrix}1&0\\0&1\end{smallmatrix}\right]\,(\mathrm{mod}\,2)\}$.
Although the half-integral weight modular forms are quite subtle objects, we can ignore such subtleties since we will only care about 4-th powers of these forms that are modular forms of weight $2$ and level $\Gamma(2)$.
We will denote them as $H_2$, $H_3$, and $H_4$, which admit Fourier expansions
\begin{align}
    H_2(z) &:= \Theta_2^4(z) = 2 \sum_{n \geq 0} r_4(2n + 1) q^{n + \frac{1}{2}}, \label{eqn:h2} \\
    H_3(z) &:= \Theta_3^4(z) = 1 + \sum_{n \geq 1}  r_4(n)q^{\frac{n}{2}}, \label{eqn:h3} \\
    H_4(z) &:= \Theta_4^4(z) = 1 + \sum_{n \geq 1} (-1)^{n} r_4(n)q^{\frac{n}{2}}, \label{eqn:h4}
\end{align}
where $r_4(k) := \# \{ \mathbf{x} \in \mathbb{Z}^{4}: \|\mathbf{x}\|^2 = k\}$.
They transform under $\mathrm{SL}_2(\mathbb{Z})$ as:
\begin{align}
    (H_{2}|_{2}T)(z) = -H_{2}(z), \quad (H_{3}|_{2}T)(z) &= H_{4}(z), \quad (H_{4}|_{2}T)(z) = H_{3}(z), \label{eqn:thetatransT} \\
    (H_{2}|_{2}S)(z) = -H_{4}(z), \quad (H_{3}|_{2}S)(z) &= -H_{3}(z), \quad (H_{4}|_{2}S)(z) = -H_{2}(z). \label{eqn:thetatransS}
\end{align}
Also, we have the Jacobi identity $H_{3} = H_{2} + H_{4}$.
These functions are related to the Eisenstein series and the discriminant form as (\cite[p. 29]{bruinier2008elliptic})
\begin{align}
    E_4 &= \frac{1}{2}(H_{2}^{2} + H_{3}^{2} + H_{4}^{2}) = H_{2}^{2} + H_{2}H_{4} + H_{4}^{2} \label{eqn:e4theta} \\
    E_6 &= \frac{1}{2} (H_{2} + H_{3})(H_{3} + H_{4}) (H_{4} - H_{2}) = \frac{1}{2}(H_2 + 2H_4)(2H_2 + H_4)(H_4 - H_2) \label{eqn:e6theta} \\
    \Delta &= \frac{1}{256} (H_{2}H_{3}H_{4})^2 = \frac{1}{256} H_2^2 (H_2 + H_4)^2 H_4^2. \label{eqn:disctheta}
\end{align}

\subsection{Serre derivative}

For a (positive) integer $k$ and a quasimodular form $F$, the weight $k$ Serre derivative $\partial_{k} F$ of $F$ is given by
\[
    \partial_{k}F := F' - \frac{k}{12} E_2 F.
\]
For $F \in \QM_{w}^{s}$, $\partial_{k}F$ is \emph{a priori} quasimodular form of weight $w + 2$ and depth $s+1$, but when $k = w - s$, Kaneko and Koike proved that it preserves depth of quasimodular forms, i.e. $\partial_{k} F \in \QM_{w+2}^{s}$ \cite[Proposition 3.3]{kaneko2006extremal}.
Serre derivative is equivariant under the $\mathrm{SL}_2(\mathbb{Z})$-action in the sense that
\[
\partial_k (F|_k \gamma) = (\partial_k F)|_{k+2}\gamma, \quad \forall \gamma \in \mathrm{SL}_2(\mathbb{Z}).
\]

Here is a list of special Serre derivatives of Eisenstein series and Jacobi thetanull functions:
\begin{equation}
    \label{eqn:serre_eis}
    \partial_{1}E_{2} = - \frac{1}{12} E_4, \quad
    \partial_{4}E_{4} = - \frac{1}{3} E_6, \quad
    \partial_{6}E_{6} = - \frac{1}{2} E_{4}^{2}
\end{equation}
and
\begin{align}
    \label{eqn:serre_jacobi}
    \partial_{2} (H_{2}) &= \frac{1}{6} (H_{2}^{2} + 2 H_{2} H_{4}), \,\,\,
    \partial_{2} (H_{3}) = \frac{1}{6} (H_{2}^{2} - H_{4}^{2}), \,\,\,
    \partial_{2} (H_{4}) = -\frac{1}{6} (2H_{2} H_{4} + H_{4}^{2}).
\end{align}
The equations \eqref{eqn:serre_eis} are the direct consequence of Ramanujan's formula \eqref{eqn:e2der}-\eqref{eqn:e4der}, and \eqref{eqn:serre_jacobi} follows from \cite[(A$\cdot$24)]{feigenbaum2021eigenfunctions}.
Serre derivative satisfies the product rule
\begin{equation}
    \label{eqn:serre_prod}
    \partial_{w_1 + w_2} (FG) = (\partial_{w_1}F)G + F(\partial_{w_2}G).
\end{equation}
Finally, we will denote the $r$-fold Serre derivatives as
\[
\partial_k^r F := \partial_{k + 2(r-1)} \partial_{k + 2(r-2)} \cdots \partial_{k} F, \quad \partial_k^0 F := F.
\]

\subsection{Linear programming bounds for optimal sphere packings and modular form inequalities}

Cohn--Elkies' linear programming bound for sphere packings reads as follows.
\begin{theorem}[{Cohn--Elkies \cite[Theorem 3.2]{cohn2003new}}]
\label{thm:sphereLP}
Let $f: \mathbb{R}^{d} \to \mathbb{R}$ be an admissible function satisfying the following three conditions for some $r > 0$:
\begin{enumerate}
    \item $f(\mathbf{0}) = \what{f}(\mathbf{0}) > 0$;
    \item $f(\mathbf{x}) \leq 0$ for $\|\mathbf{x}\| \geq r$;
    \item $\what{f}(\mathbf{x}) \geq 0$ for all $\mathbf{x} \in \mathbb{R}^{d}$.
\end{enumerate}
Then the optimal density $\Delta_d$ of sphere packings in $\mathbb{R}^{d}$ is bounded above by
\[
\Delta_d \leq \left(\frac{r}{2}\right)^{d} \frac{\pi^{d/2}}{(d/2)!}.
\]
\end{theorem}
Hence one can prove optimality of certain sphere packings by finding \emph{magic functions} with the corresponding radius $r$.
Based on their numerical experiments, Cohn and Elkies conjectured the existence of magic functions in dimension $2$, $8$, and $24$ giving the optimal bounds corresponding to the best known packings.
After a few years, Viazovska \cite{viazovska2017sphere} found an elegant construction of a magic function in dimension $8$ that proves optimality of the $E_8$ lattice sphere packing, and it only took one more week for her and other colleagues to find a similar magic function in dimension $24$ (for Leech lattice) \cite{cohn2017sphere}.
Existence of a magic function in dimension $2$ is still wide open.

Viazovska and Cohn et al.'s constructions are based on ingenious use of (quasi)modular forms.
They first decompose it as a sum of $(\pm 1)$-Fourier eigenfunctions $f = f_+ + f_-$ (hence $\what{f}_{+} = f_{+}$ and $\what{f}_{-} = - f_{-}$), and assume that they have a following form:
\[
    f_{\pm}(\mathbf{x}) = \sin^{2}\left(\frac{\pi\|\mathbf{x}\|^2}{2}\right) \int_{0}^{\infty} \varphi_{\pm}(it) e^{-\pi \|\mathbf{x}\|^{2} t} \dd t
\]
for $x \in \bR^{d}$, where $\varphi_{\pm}$ is a function defined on $\bH$.
Here we can assume that $f$ is radial (by averaging over spheres centered at the origin), and the $\sin^{2}$ factor is added to enforce roots at the desired radii, i.e. the lengths of the vectors in $E_8$ or Leech lattices
(if $f$ is a magic function, then both $f$ and $\what{f}$ should have zeros at the nonzero lattice points, which follows from the proof of Theorem \ref{thm:sphereLP}).
Surprisingly, under this assumption, they proved that $f_{\pm}$ is a ($\pm 1$)-Fourier eigenfunction if and only if $\varphi_{\pm}$ behaves as ``modular forms''.
For example, when $d = 8$, we have $\varphi_{\pm}(t) = t^{2}\psi_{\pm}(i/t)$  with\footnote{Here we normalized in a slightly different way. We have $f(0) = \what{f}(0) = \frac{5}{4 \pi}$. This normalization will be also used in Section \ref{sec:8}.}
\begin{align*}
    \psi_{+} &= -\phi_0 = - \frac{(E_2 E_4 - E_6)^2}{\Delta}, \\
    \psi_{-} &= \frac{36}{\pi^2} \psi_S = - \frac{18}{\pi^2}\frac{\Theta_{2}^{12} (2 \Theta_{2}^{8} + 5 \Theta_{2}^{4} \Theta_{4}^{4} + 5 \Theta_{4}^{8})}{\Delta}.
\end{align*}
The corresponding integrals only converge for $\|\mathbf{x}\| > \sqrt{2}$, and one needs to analytically continue to $0 \leq \|\mathbf{x}\| \leq \sqrt{2}$.
Under this construction, the corresponding inequalities for $f$ and $\what{f}$ reduce to the inequalities between $\psi_{+}$ and $\psi_{-}$, which are essentially the inequalities \eqref{eqn:d8ineq1} and \eqref{eqn:d8ineq2}.
Viazovska proved the inequalities by approximating the functions and reducing it to the finite calculations, which can be checked by interval arithmetic with computer calculations.
$d = 24$ case is similar but more involved, and requires more complicated and careful computer analysis (e.g. Sturm's theorem).

In \cite{romik2023viazovska}, Romik proposed an alternative proof of \eqref{eqn:d8ineq1} and \eqref{eqn:d8ineq2}.
His proof of \eqref{eqn:d8ineq2} uses explicit values of modular forms at $z = i$, such as
\begin{equation}
    \label{eqn:eisval}
    E_2(i) = \frac{3}{\pi}, \quad E_4(i) = \frac{3\Gamma(1/4)^8}{64 \pi^6}, \quad E_6(i) = 0.
\end{equation}
(Note that the first and third identity directly follows from \eqref{eqn:e2trans} and \eqref{eqn:e6trans}, respectively.)
He treated $0 < t < 1$ and $t \geq 1$ cases separately.
Proofs for both cases utilize monotonicity of certain modular forms on the imaginary axis and reducing the inequalities to comparisons of the values of the functions at $t = 1$.
This gives a simpler proof of the inequality \eqref{eqn:d8ineq2} that requires no interval arithmetic.

\section{Positive quasimodular forms}
\label{sec:posqmf}

In this section, we study properties of quasimodular forms that are positive on the imaginary axis or have nonnegative Fourier coefficients.
We say that a quasimodular form $F \in \QM_{w}^{s}$ is \emph{positive} if it takes positive real values on the (positive) imaginary axis ($F(it) > 0$ for all $t > 0$).
We denote as $\QM_{w}^{s,+} \subset \QM_{w}^{s}$ for the subset of positive quasimodular forms.
We define \emph{completely positive} quasimodular forms as those with nonnegative (real) Fourier coefficients, i.e. $F(z) = \sum_{n \geq n_0} a_n q^n$ with $a_n \geq 0$ for all $n \geq n_0$.
We denote the set of such forms of weight $w$ and depth $s$ as $\QM_{w}^{s, ++}$.
Both $\QM_{w}^{s,+}$ and $\QM_{w}^{s,++}$ are \emph{convex cones} in $\QM_{w}^{s}$: they are closed under nonnegative linear combinations.
Clearly we have $\QM_{w}^{s, ++} \subset \QM_{w}^{s, +}$, and the inclusion is strict in general (e.g. $\Delta(z) = q \prod_{n\geq 1} (1 - q^n)^{24}\in \QM_{12}^{0, +}$ but $\Delta(z) = q - 24 q^2 + 252q^3 - 1472 q^4 + \cdots \not \in \QM_{12}^{0, ++}$).

\subsection{Derivative and positivity}

Differentiation preserves complete positivity of quasimodular cusp forms, since it changes the $n$-th Fourier coefficient from $a_n$ to $n a_n$.
\begin{proposition}
\label{prop:dercompletepos}
Let $F \in \QM_{w}^{ s}$ and $F' \in \QM_{w+2}^{ s + 1}$.
Assume $F$ is a cusp form.
Then $F \in \QM_{w}^{ s, ++}$ if and only if $F' \in \QM_{w+2}^{ s +1, ++}$.
\end{proposition}

Positivity (not necessarily complete) is not preserved under derivatives in general.
For example, the discriminant form $\Delta$ is positive due to its product expansion, but its derivative $\Delta' = E_2 \Delta$ has a unique simple zero on the imaginary axis ($\lim_{t \to \infty}E_{2}(it) = 1$ and $\lim_{t \to 0^+} E_{2}(it) = -\infty$).
However, the positivity is preserved under \emph{antiderivatives.}
\begin{proposition}
\label{prop:dpos}
Let $F \in \QM_{w}^{ s}$ be a cusp form.
If $F' \in \QM_{w +2}^{s +1, +}$, then $F \in \QM_{w}^{s, +}$.
\end{proposition}
\begin{proof}
Let $f(t):= F(it)$ for $t > 0$.
If $F' \in \QM_{w+2}^{s +1, +}$, then $\frac{\dd f}{\dd t} = -2 \pi F'(it) < 0$ and $f(t)$ is monotone decreasing for all $t$.
Hence $f(t) > \lim_{u\to\infty} f(u) = 0$.
\end{proof}

Complete positivity can be characterized as positivity of (higher) derivatives.
To prove this, we need the following version of Bernstein's theorem.\footnote{This theorem was the motivation for the naming of \emph{complete} positivity.}

\begin{theorem}[Bernstein]
\label{thm:bernstein}
Let $g: (0, \infty) \to \mathbb{R}$ be a smooth function.
Then the following are equivalent:
\begin{enumerate}
    \item $g(t)$ is a \emph{completely monotone function}, i.e. $(-1)^{n} g^{(n)}(t) \geq 0$ for all $n \geq 0$ and $t > 0$.
    \item There exists a unique nonnegative measure $\mu$ on $(0, \infty)$ such that
    \[
        g(t) = \int_{0}^{\infty} e^{-tu} \dd \mu(u)
    \]
    where the integral converges absolutely.
\end{enumerate}
\end{theorem}
A proof can be found in \cite[Theorem 4.8, p. 40]{schilling2009bernstein}.
Note that the original version of the Bernstein's theorem consider functions on $[0, \infty)$ and finite measures, and this is a slightly generalized version of it.

\begin{proposition}
\label{prop:composder}
A cusp form $F \in \QM_{w}^{s}$ is completely positive if and only if its derivatives are all positive.
\end{proposition}
\begin{proof}
This is a direct corollary of Theorem \ref{thm:bernstein}.
Observe that $\frac{\dd^{k}}{\dd t^{k}} F(it) = (-2 \pi)^{k}F^{(k)}(it)$, and the series $f(t) = \sum_{n\geq n_0} a_n e^{-2 \pi n t}$ is the Laplace transform of the measure $\mu = \sum_{n \geq n_0} a_n \delta_{2 \pi n}$.
Although the measure $\mu$ is not finite, the integral
\[
f(t) = \int_{0}^{\infty} e^{-tu} \dd \mu(u)
\]
converges absolutely (by the polynomial growth of the coefficients \cite[Theorem 5, p. 94]{serre2012course}) and we can apply Theorem \ref{thm:bernstein}. Note that $\mu$ is supported on $(0, \infty)$ since $F$ is a cusp form.
\end{proof}

\subsection{Serre derivative and positivity}

We can also prove that \emph{anti-}Serre derivative preserves positivity, which is a simple but surprisingly powerful theorem.

\begin{proposition}
\label{prop:serrepos}
Let $F \in \QM_{w}^{s}$ be a quasimodular form.
Assume that there exists $k$ and $t_0 > 0$ such that $(\partial_{k} F)(it) > 0$ for  all $0 < t < t_0$ and $F(it_0) > 0$.
Then $F(it) > 0$ for all $0 < t \leq t_0$.
\end{proposition}

\begin{proof}
By $\Delta' = E_{2} \Delta$, we have
\begin{align*}
\frac{\dd}{\dd t} \left( \frac{F(it)}{\Delta(it)^{\frac{k}{12}}}\right) 
= (-2 \pi) \frac{F'(it) \Delta(it)^{\frac{k}{12}} - F(it) \frac{k}{12} E_{2}(it) \Delta(it)^{\frac{k}{12}}}{\Delta(it)^{\frac{k}{6}}} 
= (-2 \pi) \frac{(\partial_{k} F)(it)}{\Delta(it)^{\frac{k}{12}}}  < 0,
\end{align*}
hence $t \mapsto F(it) / \Delta(it)^{\frac{k}{12}}$ is monotone decreasing and the result follows.
\end{proof}

\begin{corollary}
\label{cor:serrepos}
Let $F \in \QM_{w}^{s}$ be a quasimodular form.
If $\partial_{k} F \in \QM_{w+2}^{s+1, +}$ and $F(it) > 0$ for \emph{sufficiently large $t > 0$}, then $F \in \QM_{w}^{s, +}$.
In particular, the assumption holds if all the Fourier coefficients of $F$ are real and the first nonzero Fourier coefficient of $F$ is positive.
\end{corollary}
\begin{proof}
First assertion directly follows from Proposition \ref{prop:serrepos}.
For the last assertion, when $F$ has a Fourier expansion $F(z) = \sum_{n \geq n_0} a_n q^{n}$ with $a_{n_0} > 0$, then
\[
e^{2 \pi n_0 t} F(it) = a_{n_{0}} + e^{-2 \pi t}\sum_{n \geq n_{0} + 1} a_n e^{-2 \pi (n - n_0 - 1)t}
\]
and $\lim_{t \to \infty} e^{2 \pi n_0 t} F(it) = a_{n_0} > 0$, so $F(it) > 0$ for sufficiently large $t$.
\end{proof}

\begin{remark}
\label{rem:serrepos}
It is also possible to \emph{solve} the differential equation $\partial_{k} F = G$ and express $f(t) = F(it)$ in  $g(t) = G(it)$ as
\begin{equation}
\label{eqn:antiserresolinit}
    f(t) = \left(\frac{\eta(it)}{\eta(it_0)}\right)^{2k} f(t_0) + 2\pi \int_{t}^{t_0} \left(\frac{\eta(it)}{\eta(iu)}\right)^{2k} g(u) \dd u.
\end{equation}
Also, Proposition \ref{prop:serrepos} holds for more general class of functions, e.g. differentiable on $(0, \infty)$ with
\[
    \partial_k := D - \frac{k}{12}E_2(it) = -\frac{1}{2 \pi} \frac{\dd}{\dd t} - \frac{k}{12}E_2(it),
\]
and this version will be used in the proof of the inequality \eqref{eqn:d24ineq3} later.
\end{remark}

In general, the Serre derivative does not preserve complete positivity, e.g. $E_4$ is completely positive but $\partial_4 E_4 = -\frac{E_{6}}{3} = -\frac{1}{3} + 168q + \cdots$ is not.
However, when the vanishing order at the cusp is sufficiently large, then it actually does.
\begin{proposition}
\label{prop:serredercompos}
Let $F = \sum_{n \geq n_0} a_{n} q^n \in \QM_{w}^{s, ++}$.
For $k \geq 0$ and $n \geq k / 12$, the $n$-th coefficient of $\partial_{k}F$ is nonnegative.
Especially, if $n_0 \geq k / 12 \geq 0$, $\partial_k F$  is also completely positive.
\end{proposition}
\begin{proof}
One can directly check that the Fourier expansion of Serre derivative is
\begin{align}
    \partial_k F &= \partial_{k} \left(a_{n_0} q^{n_{0}} + a_{n_0 + 1}q^{n_0 + 1} + a_{n_0 + 2}q^{n_0 + 2} + \cdots\right) \nonumber \\
    &= \left( n_0 a_{n_0}q^{n_0} + (n_0 + 1) a_{n_0 + 1}q^{n_0 + 1} + \cdots \right) \nonumber \\
    & \qquad - \frac{k}{12} (1 - 24q - 72q^2 - 96q^3 - \cdots) \left(a_{n_0} q^{n_{0}} + a_{n_0 + 1}q^{n_0 + 1} + a_{n_0 + 2}q^{n_0 + 2} + \cdots\right) \nonumber \\
    &= \left( n_0 - \frac{k}{12} \right) a_{n_0}q^{n_0} + \left(\left(n_0 + 1 - \frac{k}{12}\right) a_{n_0 + 1} + 2ka_{n_0}\right) q^{n_0 + 1} + \cdots \nonumber \\
    &\qquad + \left(\left(n_0 + m - \frac{k}{12}\right)a_{n_0 + m} 
 + 2k \sum_{j=1}^{m} \sigma_1(m + 1 - j)a_{n_0 + j - 1}\right)q^{n_0 + m} + \cdots .\label{eqn:serrederqexp}
\end{align}
Hence if $n_0 \geq k / 12$ and $a_j \geq 0$ for all $j \geq n_0$, the Fourier coefficients of $\partial_k F$ are also all nonnegative.
\end{proof}

\subsection{Level and positivity}

We also consider (completely) positive quasimodular forms of higher level.
For completely positive forms, we will only consider the $q$-expansions at the cusp $i\infty$, although there are several cusps for a congruence subgroup $\Gamma \subset \SL_2(\bZ)$ in general.
One easy way to construct (completely) positive quasimodular forms of level $\Gamma_0(N)$ is using \emph{old forms}.

\begin{proposition}
\label{prop:poslevelincrease}
Let $F(z) \in \QM_{w}^{s}(\SL_2(\bZ))$ be a positive (resp. completely positive) quasimodular form of weight $w$, depth $s$, and level $\SL_2(\bZ)$. Then for any $N \in \bZ_{\geq 1}$, the form $G(z) := F(Nz) \in \QM_{w}^{s}(\Gamma_0(N))$ is also a positive (resp. completely positive) quasimodular form.
\end{proposition}

\begin{proof}
If $f$ has a $q$-expansion $f(z) = \sum_{n \geq n_0} a_n q^{n}$, then $g(it) = f(iNt)$ and $g(z) = \sum_{n\geq n_0}a_n q^{Nn}$, and the proposition immediately follows.
\end{proof}

\section{Extremal quasimodular forms}
\label{sec:extremal}

\subsection{Definitions and examples}
\label{subsec:extremal_defex}

In \cite{kaneko2006extremal}, Kaneko and Koike defined and studied \emph{extremal quasimodular forms}, which are the quasimodular forms of given depth with maximum possible order of zeros at infinity.
In other words, for given weight $w$ and depth $s$, a quasimodular form $f \in \QM_{w}^{ s} \backslash \QM_{w}^{s - 1}$ is \emph{extremal} if, for $m = \dim_\mathbb{C} \QM_{w}^{s}$, the first $m$ Fourier coefficients of $f = \sum_{n \geq 0} a_n q^n$ are 
\[
    a_0 = a_1 = \cdots = a_{m-2} = 0, a_{m-1} \neq 0.
\]
The authors conjectured existence and uniqueness (up to a constant) of extremal forms for each (even) weight $w$ and depth $s$ (satisfying $0\leq s\leq w/2, s \neq \frac{w}{2} - 1$), and give examples in case of depth $1$ and $2$ that are defined recursively and satisfying certain differential equations.
Pellarin \cite[Theorem 1.3, p. 403]{pellarin2020extremal} established the conjecture for $s \leq 4$, and
Grabner \cite{grabner2020quasimodular} extended Kaneko--Koike's result and constructed differential equations satisfied by depth $\leq 4$ extremal quasimodular forms.
For these depths, we will denote the normalized (i.e. the first nonzero Fourier coefficient is $1$) extremal form of weight $w$ and depth $s$ as $X_{w, s}$.
Kaneko and Koike also conjectured that the Fourier coefficients of extremal forms of depth $\leq 4$ are all positive \cite[Conjecture 2, p. 469]{kaneko2006extremal}, and Grabner \cite[Theorem 1, p. 1030]{grabner2022asymptotic} proved the conjecture \emph{for all but finitely many coefficients}.
The proof uses Jenkins and Rouse's explicit version of Deligne's bound \cite{jenkins2011bounds}.

We can check complete positivity of certain low-weight extremal forms from Ramanujan's identities:
\begin{lemma}
\label{lem:lowpos}
The following quasimodular forms are extremal and completely positive:
\begin{align*}
    X_{4, 2} &= \frac{1}{288}(E_4 - E_2^2) \in \QM_{4}^{2, ++}, \\
    X_{6, 1} &= \frac{1}{720}(E_2E_4 - E_6) \in \QM_{6}^{1, ++}, \\
    X_{8, 1} &= \frac{1}{1008}(E_4^2 - E_2E_6) \in \QM_{8}^{1, ++}.
\end{align*}
\end{lemma}
\begin{proof}
Extremalities are mentioned in \cite[Example 1.4, p. 459]{kaneko2006extremal}.
Complete positivity directly follows from Ramanujan's identities,
\begin{align*}
    X_{4, 2} = \frac{1}{288}(E_4 - E_2^2) &= -\frac{1}{24} E_2' = \sum_{n\geq 1} n\sigma_1(n) q^n, \\
    X_{6, 1} = \frac{1}{720}(E_2 E_4 - E_6) &= \frac{1}{240}E_4' =  \sum_{n\geq 1} n\sigma_3(n) q^n, \\
    X_{8, 1} = \frac{1}{1008}(E_4^2 - E_2 E_6) &= -\frac{1}{504}E_6' = \sum_{n\geq 1} n\sigma_5(n) q^n.
\end{align*}
\end{proof}
More examples can be found in Appendix \ref{subsec:exttab}, Table \ref{tab:extform}.

\subsection{Kaneko--Koike's conjecture for depth $1$}
In this subsection, we prove the conjecture of Kaneko and Koike \cite[Conjecture 2, p. 469]{kaneko2006extremal} in the case of depth $1$.

For even $w \geq 6$, let $X_{w} = X_{w, 1}$ be the unique normalized extremal quasimodular form of weight $w$ and depth $1$.
We have $X_{6} = (E_2 E_4 - E_6) / 720$ and $X_w$'s satisfy the following recursive formula for $w \geq 6$ with $6|w$ \cite[Proposition 6.1, p. 2258]{grabner2020quasimodular}: 
\begin{align}
    X_{w + 2} &= \frac{12}{w + 1} \partial_{w - 1}X_{w} \label{eqn:d1rec1}\\
    X_{w + 4} &= E_4 X_{w} \label{eqn:d1rec2}\\
    X_{w + 6} &= \frac{w + 6}{72(w + 1)(w + 5)} \left(E_4 \partial_{w - 1} X_{w} - \frac{w + 1}{12} E_6 X_{w}\right)  \nonumber \\
    &= \frac{w + 6}{864(w + 5)} \left(E_4 X_{w + 2} - E_6 X_{w}\right). \label{eqn:d1rec3}
\end{align}
(The last equality follows from \cite[eq. (6.6), p. 2258]{grabner2020quasimodular} and \eqref{eqn:d1rec1}.)
The vanishing order of $X_{w}$ at the cusp is $\lfloor \frac{w}{6} \rfloor$.
Also, $X_{w}$ is a solution of the following differential equation (for $6 \mid w$  \cite[Theorem 2.1 (1), p. 461]{kaneko2006extremal})
\begin{equation}
\label{eqn:oded1}
X_{w}'' - \frac{w}{6}E_2 X_{w}' + \frac{w(w - 1)}{144} (E_2^2 - E_4) X_{w} = 0
\end{equation}
or equivalently (see \cite[eq. (6.1), p. 2257]{grabner2020quasimodular}),
\begin{equation}
    \partial_{w-1}^2 X_{w} - \frac{w^2 - 1}{144} E_4 X_w = 0 \label{eqn:oded1serre}
\end{equation}

Our main goal is to prove new recurrence relation \eqref{eqn:kkd1eq1}, which immediately implies the complete positivity of $X_{w}$'s.
In the proofs, we adopt the notation used in \cite{grabner2020quasimodular}, where an equation number with subscript \(w+a\) means that the parameter \(w\) in the original equation is replaced by \(w + a\).

\begin{lemma}
    \label{lem:kkd1}
    For $6 \mid w$, we have
    \begin{equation}
        X_{w+4} = \frac{12}{w-1} \partial_{w+1} X_{w+2} = \frac{12}{w-1}\left(X_{w+2}' - \frac{w+1}{12} E_2 X_{w+2}\right). \label{eqn:d1rec2serre}
    \end{equation}
\end{lemma}
\begin{proof}
    The first equality follows from \eqref{eqn:d1rec1} and \eqref{eqn:oded1serre}:  
    \begin{equation}
        \label{eqn:d1rec4}
        X_{w+4} = E_4 X_w = \frac{144}{(w-1)(w+1)} \partial_{w-1}^2 X_w = \frac{12}{w-1} \partial_{w+1} X_{w+2}.
    \end{equation}
    Second equality follows from the definition of Serre derivative.
\end{proof}
\begin{theorem}
\label{thm:kkd1}
Let $X_{w} = X_{w, 1}$ be the unique normalized extremal quasimodular form of weight $w$ and depth $1$.
For $6|w$ and  $w \geq 12$, we have
\begin{equation}
    X_{w}' = \frac{5w}{72} X_{6} X_{w-4} + \frac{7w}{72} X_{8} X_{w - 6}. \label{eqn:kkd1eq1}
\end{equation}
\end{theorem}

\begin{proof}
From \eqref{eqn:d1rec2serre}${}_{w-6}$ and \eqref{eqn:d1rec2}${}_{w-6}$, we get
\begin{equation}
X_{w-4}' = \frac{w-5}{12} E_2 X_{w-4} + \frac{w-7}{12} E_4 X_{w-6} \label{eqn:d1eq}
\end{equation}
and differentiating \eqref{eqn:d1rec3}$_{w-6}$ gives
\begin{alignat*}{2}
X_{w}' &= \frac{w}{864(w - 1)} (E_4' X_{w - 4} + E_4 X_{w - 4}' - E_6' X_{w - 6} - E_6 X_{w - 6}') \qquad &&\text{$\cdots$\eqref{eqn:d1rec3}${}_{w-6}$}\\
&= \frac{w}{864(w -1)} \left(\frac{E_2 E_4 - E_6}{3}X_{w - 4} + \frac{w - 5}{12}E_2 E_4 X_{w - 4} + \frac{w - 7}{12}E_4^2 X_{w - 6} \right. \qquad && \text{$\cdots$\eqref{eqn:d1eq}${}_{w}$}\\
&\quad \left.- \frac{E_2 E_6 - E_4^2}{2}X_{w - 6} - E_6 \left(\frac{w - 5}{12}X_{w - 4} + \frac{w - 7}{12}E_2 X_{w 
- 6}\right)\right) \qquad && \text{$\cdots$\eqref{eqn:d1rec1}${}_{w-6}$}\\
&= \frac{w}{864(w - 1)} \left(\frac{w - 1}{12}(E_2 E_4 - E_6) X_{w - 4} + \frac{w - 1}{12}(E_4^2 - E_2 E_6)X_{w - 6}\right) \\
&= \frac{5w}{72} X_6 X_{w - 4} + \frac{7w}{72}X_8 X_{w - 6}.
\end{alignat*}
\end{proof}

\begin{corollary}
\label{cor:kkd1}
The Kaneko--Koike's conjecture is true for depth $1$ extremal forms.
\end{corollary}
\begin{proof}
The conjecture holds for $X_{6}$ and $X_{8}$ by Lemma \ref{lem:lowpos}, and $X_{10} = E_{4} X_{6}$ shows that $X_{10}$ is also completely positive.
Now, let $w \ge 12$ be an integer with $6 \mid w$ and assume that $X_{k}$ is completely positive for $k \leq w-2$.
Then \eqref{eqn:kkd1eq1}${}_{w}$ and Proposition \ref{prop:dercompletepos} imply that $X_{w}$ is also completely positive.
Combining it with Proposition \ref{prop:serredercompos} (recall that the vanishing order of $X_{w}$ at the cusp is $\frac{w}{6} > \frac{w-1}{12}$) and \eqref{eqn:d1rec1}${}_{w}$ shows $X_{w+2} \in \QM_{w+2}^{1,++}$, and we also get $X_{w+4} \in \QM_{w+4}^{1,++}$ from \eqref{eqn:d1rec2}, since $E_4$ has nonnegative Fourier coefficients.
\end{proof}

\begin{remark}
We also have the following relations (for $6|w$):
\begin{align}
    X_{w + 2}' &= \frac{5w}{72} X_{6} X_{w-2} + \frac{7w}{12} X_{8}X_{w-4}, \label{eqn:kkd1eq2}\\
    X_{w + 4}' &= 240 X_{6} X_{w} + \frac{7w}{72} X_{8}X_{w-2} + \frac{5w}{72} X_{10}X_{w-4} \label{eqn:kkd1eq3}
\end{align}
which can be proven similarly and also provide an alternative proof of Corollary \ref{cor:kkd1}.
\end{remark}

\subsection{Extremal forms of depth $2$}

For even $w \geq 4$ and $w\neq 6$, the depth $2$ (normalized) extremal forms $X_{w, 2}$ satisfy the following recurrence relations \cite[Proposition 6.2, p. 2259]{grabner2020quasimodular}\footnote{There's a minor error in \cite{grabner2020quasimodular}.
We need to replace $w^2$ with $(w+4)^{2}$ in the numerator to make sure that $X_{w+4, 2}$ is normalized, i.e. its leading coefficient is $1$.
We fix this in \eqref{eqn:d2eq1}.}: $X_{4, 2} = \frac{E_4 - E_2^2}{288}$ and for $4 \mid w$,
\begin{align}
    X_{w+4, 2} &= \frac{3(w+4)^2}{16(w+1)(w+2)^{2}(w+3)} \left(\frac{w(w+1)}{36} E_4 X_{w, 2} - \partial_{w-2}^{2}X_{w, 2}\right) \label{eqn:d2eq1} \\
    X_{w+2, 2} &= \frac{6}{w + 1} \partial_{w-2}X_{w, 2} \label{eqn:d2eq2} \\
    &= \frac{3w^2}{16(w^2 - 1)(w-6)^2} \left(\frac{(w-4)(w-5)}{36} E_4 X_{w-2,2} - \partial_{w-4}^{2}X_{w-2, 2}\right) \label{eqn:d2eq3}
\end{align} 
The vanishing order of $X_{w, 2}$ at the cusp is $\lfloor \frac{w}{4} \rfloor$.
Also, $X_{w,2}$ (for $4 \mid w$) is a solution of the differential equation
\begin{equation}
\label{eqn:d2ode2}
X_{w, 2}''' - \frac{w}{4} E_2 X_{w, 2}'' + \frac{w(w-1)}{4} E_{2}'X_{w, 2}' - \frac{w(w-1)(w-2)}{24} E_2'' X_{w, 2} = 0.
\end{equation}

We found the following (exceptional) identities that verify the conjecture for depth $2$ and weight $\leq 14$ that can be checked directly, although we could not find similar recurrence relations as \eqref{eqn:kkd1eq1} in the case of depth $2$ that may prove the conjecture completely.

\begin{proposition}
\label{prop:d2extposexceptional}
We have the following identities:
\begin{align}
    X_{8, 2}' &=  2X_{4, 2} X_{6, 1}, \label{eqn:dx82} \\
    X_{10, 2}' &=  \frac{8}{9} X_{4, 2} X_{8, 1} + \frac{10}{9} X_{6, 1}^2, \label{eqn:dx102} \\
    X_{12, 2}' &= 3 X_{6, 1} X_{8, 2}, \label{eqn:dx122} \\
    X_{14, 2}' &= 3 X_{4, 2} X_{12, 1}. \label{eqn:dx142}
\end{align}
Especially, $X_{w, 2}$ is completely positive for $w \leq 14$.
\end{proposition}
\begin{proof}
These identities can be checked directly by expressing them as polynomials in $E_2, E_4, E_6$ and using Ramanujan's formula \eqref{eqn:e2der}-\eqref{eqn:e6der}; see also Appendix \ref{subsec:appendix_xd2}.
Complete positivity follows from the identities and Proposition \ref{prop:dercompletepos}.
\end{proof}

\section{8-dimensional inequalities}
\label{sec:8}

Now we are ready to give a new proof of Theorem \ref{thm:d8ineq}. 
Define
\begin{align}
    F(z) &= (E_2(z) E_4(z) - E_6(z))^2 \label{eqn:d8f} \\
    G(z) &= H_{2}(z)^{3} (2H_{2}(z)^{2} + 5H_{2}(z) H_{4}(z) + 5 H_{4}(z)^{2}). \label{eqn:d8g}
\end{align} 
Then we can check $F(z) = \Delta(z) \phi_0(z)$ and $G(z) = - 2 \Delta(z) \psi_S(z)$,
where the second identity on $G(z)$ follows from the Jacobi identity and \eqref{eqn:disctheta}.
Since $\Delta(it) > 0$, the inequalities \eqref{eqn:d8ineq1} and \eqref{eqn:d8ineq2} are equivalent to
\begin{align}
    F(it) + \frac{18}{\pi^2}G(it) &> 0, \label{eqn:d8ineq1new} \\
    F(it) - \frac{18}{\pi^2}G(it) &< 0. \label{eqn:d8ineq2new}
\end{align}
As already mentioned in \cite{romik2023viazovska}, $F(it) = 9 E_4'(it)^2 = 9 (240\sum_{n\geq 1}n \sigma_3(n) e^{-2 \pi n t})^2 > 0$, and $G(it) > 0$ follows from $\Theta_2(it) > 0$ and $\Theta_4(it) > 0$.
Note that $F$ can be also written as $F = 720^{2}X_{6, 1}^{2}$.

As we mentioned before, our proof of the ``hard'' inequality \eqref{eqn:d8ineq2new} is based on the following observations (Figure \ref{fig:d8graph}).

\begin{proposition}
\label{prop:d8limit}
\begin{equation}
    \label{eqn:d8lim}
    \lim_{t \to 0^+} \frac{F(it)}{G(it)} = \frac{18}{\pi^2}.
\end{equation}
\end{proposition}

\begin{proposition}
\label{prop:d8decrease}
The function $t \mapsto \frac{F(it)}{G(it)}$ is strictly decreasing on $t > 0$.
\end{proposition}

It is clear that the inequality \eqref{eqn:d8ineq2new} follows from Proposition \ref{prop:d8limit} and \ref{prop:d8decrease}.

\begin{proof}[Proof of Proposition \ref{prop:d8limit}]
We have
\[
\lim_{t \to 0^+} \frac{F(it)}{G(it)} = \lim_{t \to \infty} \frac{F(i/t)}{G(i/t)}.
\]
By using the transformation laws of Eisenstein series and the thetanull functions \eqref{eqn:e2trans}-\eqref{eqn:e6trans} and \eqref{eqn:thetatransS}, we get
\begin{align*}
    F\left(\frac{i}{t}\right) &= t^{12} F(it) - \frac{12t^{11}}{\pi} (E_2(it)E_4(it) - E_6(it))E_4(it) + \frac{36t^{10}}{\pi^2}E_4(it)^2, \\
    G\left(\frac{i}{t}\right) &= t^{10} H_{4}(it)^{3}(2H_{4}(it)^{2} + 5 H_{4}(it)H_{2}(it) + 5 H_{2}(it)^{2}).
\end{align*}
Since $F$, $E_2 E_4 - E_6$ and $H_2$ are cusp forms, we have $\lim_{t \to \infty}t^k A(it) = 0$ when $A(z)$ is one of these forms and $k \geq 0$.
From $\lim_{t \to \infty} E_4(it) = 1 = \lim_{t \to \infty}H_{4}(it)$, we get
\begin{align*}
    \lim_{t \to \infty} \frac{F(i/t)}{G(i/t)}
    &= \lim_{t \to \infty} \frac{t^{12} F(it) - \frac{12t^{11}}{\pi} (E_2(it)E_4(it) - E_6(it))E_4(it) + \frac{36t^{10}}{\pi^2}E_4(it)^2}{t^{10} H_{4}(it)^{3}(2H_{4}(it)^{2} + 5 H_{4}(it)H_{2}(it) + 5 H_{2}(it)^{2})} \\
    &= \lim_{t \to \infty} \frac{t^{2} F(it) - \frac{12t}{\pi} (E_2(it)E_4(it) - E_6(it))E_4(it) + \frac{36}{\pi^2}E_4(it)^2}{H_{4}(it)^{3}(2H_{4}(it)^{2} + 5 H_{4}(it)H_{2}(it) + 5 H_{2}(it)^{2})} \\
    &= \frac{18}{\pi^2}.
\end{align*}
\end{proof}

\begin{proof}[Proof of Proposition \ref{prop:d8decrease}]
It is enough to show that the derivative
\[
\frac{\dd}{\dd t} \left(\frac{F(it)}{G(it)}\right) = -2\pi \frac{F'(it)G(it) - F(it)G'(it)}{G(it)^2}
\]
is negative, which is equivalent to the positivity of (recall that $F(it) > 0$ and $G(it) > 0$)
\begin{equation}
    F' G - F G' = (\partial_{10}F) G - F (\partial_{10}G) =: \cL_{1, 0} \label{eqn:ineq3}
\end{equation}
that has weight $24$, level $\Gamma(2)$, and depth $2$.
$F$ and $G$ satisfy the following differential identities:
\begin{align}
    \partial_{10}^{2} F &= \frac{5}{6} E_4 F + a \Delta X_{4, 2}, \label{eqn:d8ssf} \\
    \partial_{10}^{2} G &= \frac{5}{6} E_4 G - b \Delta H_{2}, \label{eqn:d8ssg}
\end{align}
where $a = 12^3 \cdot 100$ and $b = 640$ (see also Appendix \ref{subsec:appendix_d8}).
Combining these with \eqref{eqn:serre_prod}, we get
\begin{equation}
    \partial_{22}\cL_{1, 0} = (\partial_{10}^{2}F) G - F (\partial_{10}^{2}G) = \Delta(a X_{4, 2} G + b H_{2} F) > 0.
\end{equation}
By Corollary \ref{cor:serrepos}, it is enough to show that $\cL_{1, 0}(it) > 0$ for sufficiently large $t > 0$.
This can be done by comparing the vanishing orders at the cusp.
From $E_2E_4 - E_6 = 3E_4' = 720 q + O(q^2)$, $H_2 = 16 q^{\frac{1}{2}} + O(q^{\frac{3}{2}})$, and $H_4 = 1 + O(q^{\frac{1}{2}})$, we have
\[
    F = 720^2 q^2 + O(q^3), \quad G = 16^3 \cdot 5 q^{\frac{3}{2}} + O(q^2)
\]
and
\[
\frac{F'}{F} = \frac{2 \cdot 720^2 q^2 + O(q^3)}{720^2 q^2 + O(q^3)} = 2 + O(q), \quad \frac{G'}{G} = \frac{\frac{3}{2} \cdot 16^3 \cdot 5 q^{\frac{3}{2}} + O(q^2)}{16^3 \cdot 5 q^{\frac{3}{2}} + O(q^2)} = \frac{3}{2} + O(q^{\frac{1}{2}}).
\]
Since \(2 > \frac{3}{2}\), we get $\frac{F'(it)}{F(it)} > \frac{G'(it)}{G(it)}$ for sufficiently large \(t > 0\), which is equivalent to \(\cL_{1, 0}(it) > 0\).
\end{proof}

\begin{remark}
For $k \in \mathbb{Z}$, let $L_{2, k} := \partial_{k}^{2} - \frac{k(k+2)}{144} E_{4}$ be the type $(k, k+4)$ modular linear differential operator defined in \cite{nagatomo2024modular} that is originated from \cite{kaneko1998supersingular}.
Then the identities \eqref{eqn:d8ssf} and \eqref{eqn:d8ssg} show $L_{2, 10}F > 0$ and $L_{2, 10}G < 0$.
\end{remark}

\begin{remark}
\label{rem:d8proof2}
An alternative proof the positivity of \(\cL_{1, 0}\) is given in Appendix \ref{subsec:appendix_d8}.
\end{remark}

\section{24-dimensional inequalities}
\label{sec:24}

We give a similar proof of Theorem \ref{thm:d24ineq}.
The \emph{easy} inequality \eqref{eqn:d24ineq1} is not as easy as \eqref{eqn:d8ineq1}, but it almost directly follows from Corollary \ref{cor:serrepos}.
The idea for the proof of the inequality \eqref{eqn:d24ineq2} is the same as that of \eqref{eqn:d8ineq2}, and \eqref{eqn:d24ineq3} is more involved.
We will abuse notation and write $F = -\Delta^2 \varphi$ and $G = -\Delta^2 \psi_S$ for the numerators of \eqref{eqn:d24forig} and \eqref{eqn:d24gorig}.
Then the inequalities \eqref{eqn:d24ineq1}, \eqref{eqn:d24ineq2} and \eqref{eqn:d24ineq3} are equivalent to
\begin{align}
    F(it) + \frac{432}{\pi^2} G(it) &> 0, \label{eqn:d24ineq1new} \\
    F(it) - \frac{432}{\pi^2} G(it) &< 0, \label{eqn:d24ineq2new} \\
    t^{10} \left(-\frac{F(i/t)}{\Delta(i/t)^2} + \frac{432}{\pi^2} \frac{G(i/t)}{\Delta(i/t)^2}\right) &\geq \frac{725760}{\pi} e^{2\pi t} \left(t - \frac{10}{3 \pi}\right). \label{eqn:d24ineq3new}
\end{align}

\subsection{``Easy'' inequality \eqref{eqn:d24ineq1new}}

As in the $d = 8$ case, we can prove \eqref{eqn:d24ineq1new} by proving the \emph{easy} inequalities $F(it) > 0$ and $G(it) > 0$ separately.
Especially, the second inequality directly follows from its expression (which is already mentioned in \cite[p. 1028]{cohn2017sphere}).
However, the other inequality $F(it) > 0$ is less trivial, although it follows from Corollary \ref{cor:serrepos} and the identity \eqref{eqn:serrederF16}.

\begin{lemma}
\label{lem:serrederF16}
We have
\begin{equation}
    \label{eqn:serrederF16}
    \partial_{14}F = 6706022400  \cdot X_{6, 1} X_{12, 1} \in \QM_{18}^{2,++}
\end{equation}
where \(X_{6, 1}\) (resp. \(X_{12, 1}\)) are normalized extremal quasimodular forms of weight 6 (resp. 12) and depth 1 (Section \ref{subsec:extremal_defex}).
\end{lemma}
\begin{proof}
See Appendix \ref{subsec:appendix_d24} for the direct proof of \eqref{eqn:serrederF16}.
(Complete) positivity follows from Corollary \ref{cor:kkd1}.
\end{proof}

\begin{corollary}
\label{cor:F16pos}
$F(it) > 0$ for all $t > 0$.
\end{corollary}
\begin{proof}
The $q$-expansion of $F$ is
\[
    F = 3657830400 q^3 + 138997555200 q^4 + 2567796940800 q^5 + O(q^6)
\]
where its first nonzero Fourier coefficient is positive.
Now the positivity follows from Lemma \ref{lem:serrederF16} and Corollary \ref{cor:serrepos}.
\end{proof}

\begin{remark}
$F$ is a constant multiple of $f_{16}$ that appears in the family of Feigenbaum--Grabner--Hardin \cite[Proposition 5.1]{feigenbaum2021eigenfunctions}.
The authors already proved that the functions $f_{w}$ are \emph{completely} positive for $w \leq 94$ \cite[Remark 6.3]{feigenbaum2021eigenfunctions}, and they conjectured that all the forms in the family are completely positive \cite[Conjecture 1]{feigenbaum2021eigenfunctions}.
However, their proof uses approximations based on Jenkins and Rouse's explicit bound on Fourier coefficients \cite{jenkins2011bounds}.
\end{remark}

\subsection{``Hard'' inequality \eqref{eqn:d24ineq2new}}
\begin{figure}
    \centering
    \includesvg[width=0.6\textwidth]{figs/d24}
    \caption{Graph of the quotient $F(it)/G(it)$ as a function in $t > 0$.}
    \label{fig:d24graph}
\end{figure}

The previous approach we used in the proof of \eqref{eqn:d8ineq2new} also works for \eqref{eqn:d24ineq2new}, as we can expect from Figure \ref{fig:d24graph}.
It follows from the following two propositions.

\begin{proposition}
\label{prop:d24limit}
\begin{equation}
    \label{eqn:d24lim}
    \lim_{t\to 0^+} \frac{F(it)}{G(it)} = \frac{432}{\pi^2}.
\end{equation}
\end{proposition}
\begin{proof}
    This can be proved similarly as Proposition \ref{prop:d8limit}.
    From the transformation laws \eqref{eqn:e2trans}-\eqref{eqn:e6trans} and \eqref{eqn:thetatransS}, we get
    \begin{align*}
        F\left(\frac{i}{t}\right) &= t^{16} F(it) - \frac{12 t^{15}}{\pi}(49 E_2 E_4^3 - 25 E_2 E_6^2 - 24 E_4^2 E_6)(it) + \frac{36 t^{14}}{\pi^2} (49 E_4^3 - 25 E_6^2)(it) \\
        G\left(\frac{i}{t}\right) &= t^{14} H_4(it)^5 (2 H_4(it)^2 + 7 H_4(it) H_2(it) + 7 H_2(it)^2)
    \end{align*}
    and
    \begin{align*}
        \lim_{t \to 0^+} \frac{F(it)}{G(it)} &= \lim_{t \to \infty} \frac{F(i/t)}{G(i/t)} \\
        &= \lim_{t \to \infty} \frac{t^2 F(it) - \frac{12 t}{\pi}(49 E_2 E_4^3 - 25 E_2 E_6^2 - 24 E_4^2 E_6)(it) + \frac{36}{\pi^2} (49 E_4^3 - 25 E_6^2)(it)}{H_4(it)^5 (2 H_4(it)^2 + 7 H_4(it) H_2(it) + 7 H_2(it)^2)}.
    \end{align*}
    Since $F$, $49 E_2 E_4^3 - 25 E_2 E_6^2 - 24 E_4^2 E_6$, and $H_2$ are cusp forms, the first two terms of the numerator and the terms including $H_2$ of the denominator do not contribute to the limit. Thus we get
    \[
        \frac{\frac{36}{\pi^2} \cdot (49 - 25)}{2} = \frac{432}{\pi^2}.
    \]
\end{proof}

\begin{proposition}
\label{prop:d24decrease}
The function $t \mapsto \frac{F(it)}{G(it)}$ is strictly decreasing on $t > 0$.
\end{proposition}

\begin{proof}
The idea is similar to that of Proposition \ref{prop:d8decrease}. It is enough to show that
\[
    F' G - F G' = (\partial_{14}F) G - F (\partial_{14}G) =: \cL_{1, 0}
\]
is positive, which has weight $32$, level $\Gamma(2)$, and depth $2$.
$F$ and $G$ satisfy the following differential equations similar to \eqref{eqn:d8ssf} and \eqref{eqn:d8ssg}:
\begin{align}
    \partial_{14}^{2} F &= \frac{14}{9} E_4 F + c \Delta X_{8, 2} \label{eqn:d24ssf} \\
    \partial_{14}^{2} G &= \frac{14}{9} E_4 G \label{eqn:d24ssg}
\end{align}
where $X_{8, 2}$ is the normalized extremal quasimodular form of weight 8 and depth 2, and $c = 2^{11} \cdot 3^7 \cdot 5^2 \cdot 7^2$ (see Appendix \ref{subsec:appendix_d24}).
If we take the Serre derivative $\partial_{30}$, by \eqref{eqn:serre_prod}, \eqref{eqn:d24ssf}, and \eqref{eqn:d24ssg}, we have
\begin{equation}
    \label{eqn:d24L20}
    \partial_{30}\cL_{1, 0} = \cL_{2, 0} := (\partial_{14}^{2} F) G - F (\partial_{14}^{2}G) = c \Delta X_{8, 2}G > 0
\end{equation}
since $X_{8, 2}$ is (completely) positive by Proposition \ref{prop:d2extposexceptional}.
As in the proof of Proposition \ref{prop:d8decrease}, to prove $\cL_{1, 0}(it) > 0$ for sufficiently large $t$, it is enough to show that the vanishing order of $F$ at infinity is larger than that of $G$.
The vanishing order of $F$ is 3 and that of $G$ is $\frac{5}{2}$, hence $\cL_{1, 0}(it)$ is positive for sufficiently large $t > 0$.
Hence Corollary \ref{cor:serrepos} applies and we get the positivity of $\cL_{1, 0}$.
\end{proof}

\begin{remark}
As in the $d = 8$ case, \eqref{eqn:d24ssf} and \eqref{eqn:d24ssg} describe the action of the operator $L_{2, 14}$ of type $(14, 18)$ in \cite{kaneko1998supersingular,nagatomo2024modular} on $F$ and $G$. Especially, $G$ is a solution of the linear differential equation $L_{2, 14} G = 0$.
\end{remark}

\subsection{``Harder'' inequality \eqref{eqn:d24ineq3new}}
\label{subsec:d24harder}

The last inequality \eqref{eqn:d24ineq3new} is more involved than the previous inequalities because of the presence of the non-modular terms $t^{10}$ and $e^{2 \pi t}$.
We first replace the exponential term $e^{2 \pi t}$ with $1 / \Delta$ using the following inequality.
\begin{lemma}
\label{lem:discineq}
For $t > 0$, we have
\begin{equation}
    \Delta(it) < e^{-2 \pi t}. \label{eqn:discineq}
\end{equation}
\end{lemma}
\begin{proof}
This directly follows from the product formula of $\Delta$,
\[
    \Delta(it) = e^{-2 \pi t} \prod_{n \geq 1} (1 - e^{-2 \pi n t})^{24} < e^{-2 \pi t}.
\]
\end{proof}
The inequality \eqref{eqn:d24ineq3new} is true for $1 \leq t \leq \frac{10}{3\pi}$, since the left hand side (resp. the right hand side) is nonnegative (resp. nonpositive) on this range (by \eqref{eqn:d24ineq2new}).
Hence it is enough to show for $t > \frac{10}{3 \pi}$.
On this range, we can bound the right hand side of \eqref{eqn:d24ineq3new} with \eqref{eqn:discineq}
\[
    \frac{725760}{\pi} e^{2 \pi t} \left(t - \frac{10}{3 \pi}\right) < \frac{725760}{\pi} \frac{1}{\Delta(it)} \left(t - \frac{10}{3 \pi}\right) = \frac{725760}{\pi} \frac{t^{12}}{\Delta(i/t)} \left(t - \frac{10}{3 \pi}\right)
\]
and the inequality reduces to
\begin{align*}
    t^{10} \left(- \frac{F(i/t)}{\Delta(i/t)^2} + \frac{432}{\pi^2} \frac{G(i/t)}{\Delta(i/t)^2}\right) &> \frac{725760}{\pi} \frac{t^{12}}{\Delta(i/t)} \left(t - \frac{10}{3 \pi}\right) \nonumber \\
    \Leftrightarrow \frac{1}{t^2} \left(- \frac{F(i/t)}{\Delta(i/t)} + \frac{432}{\pi^2} \frac{G(i/t)}{\Delta(i/t)}\right) &> \frac{725760}{\pi} \left(t - \frac{10}{3 \pi}\right)
\end{align*}
for $t > \frac{10}{3 \pi}$.
If we replace $t$ by $1/t$, the last inequality becomes
\begin{equation}
    t^{3} \left( - \frac{F(it)}{\Delta(it)} + \frac{432}{\pi^2} \frac{G(it)}{\Delta(it)} \right) > \frac{725760}{\pi} \left(1 - \frac{10 t}{3 \pi} \right)
    \Leftrightarrow \frac{432}{\pi^2} - \frac{F(it)}{G(it)} >  \frac{725760\Delta(it)}{G(it)} \left(\frac{1}{\pi t^3} - \frac{10}{3 \pi^2 t^2}\right) \label{eqn:d24ineq3new2}
\end{equation}
for $0 < t < \frac{3\pi}{10}$.
From Proposition \ref{prop:d24decrease}, we know that the left hand side of \eqref{eqn:d24ineq3new2} is monotone increasing in $t$.
Our main observation is that the difference between two sides of \eqref{eqn:d24ineq3new2} is also monotone increasing (Figure \ref{fig:d24harder_mono}).

\begin{figure}[t]
    \centering
    \includesvg[width=0.6\textwidth]{figs/d24harder_mono.svg}
    \caption{Graph of $\mathrm{LHS}(t)$, $\mathrm{RHS}(t)$, and $g(t) = \mathrm{LHS}(t) - \mathrm{RHS}(t)$ of \eqref{eqn:d24ineq3new2}.}
    \label{fig:d24harder_mono}
\end{figure}

\begin{proposition}
\label{prop:d24harder}
The function
\[
    g(t) := \frac{432}{\pi^2} - \frac{F(it)}{G(it)} - \frac{725760\Delta(it)}{G(it)} \left(\frac{1}{\pi t^3} - \frac{10}{3 \pi^2 t^2}\right) 
\]
is monotone increasing in $t$ for $0 < t < \frac{3\pi}{10}$ and $\lim_{t \to 0^+} g(t) = 0$.
Especially, we have $g(t) > 0$ for all $0 < t < \frac{3\pi}{10}$.
\end{proposition}

\begin{proof}
After writing the limit as $\lim_{t \to 0^+} g(t) = \lim_{t \to \infty} g(1/t)$, we can compute the limit from Proposition \ref{prop:d24limit} and the fact that $\Delta|_{12}S = \Delta$ is a cusp form but $G|_{14}S = -H_4^5(7 H_2^2 + 7 H_2 H_4 + 2 H_4^2)$ is not, so the third term in $g(t)$ vanishes as $t \to 0^+$.
We will omit the details and focus on the monotonicity part.
We have
\[
    \frac{\dd}{\dd t} \left(\frac{F(it)}{G(it)}\right) = -2 \pi \frac{\cL_{1, 0}(it)}{G(it)^2}
\]
and by $\Delta' = E_2 \Delta$,
\begin{align*}
    \frac{\dd}{\dd t} \left[\frac{\Delta(it)}{G(it)} \left(\frac{1}{\pi t^3} - \frac{10}{3 \pi^2 t^2}\right)\right] 
    &= (-2\pi) \frac{\Delta(it)(E_2(it)G(it) - G'(it))}{G(it)^2} \left(\frac{1}{\pi t^3} - \frac{10}{3 \pi^2 t^2}\right) + \frac{\Delta(it)}{G(it)} \left(-\frac{3}{\pi t^4} + \frac{20}{3 \pi^2 t^3}\right) \\
    &= (2\pi) \frac{\Delta(it)}{G(it)^2} \left[(\partial_{12}G)(it) \left(\frac{1}{\pi t^3} - \frac{10}{3 \pi^2 t^2}\right) - G(it) \left(\frac{3}{2\pi^2 t^4} - \frac{10}{3 \pi^3 t^3}\right) \right],
\end{align*}
so $\dd g / \dd t > 0$ if and only if (after factoring out $1 / G^2$)
\begin{small}
\[
    \widetilde{\cL}_{1, 0}(it) := \cL_{1, 0}(it) - 725760\Delta(it) \left[(\partial_{12}G)(it) \left(\frac{1}{\pi t^3} - \frac{10}{3 \pi^2 t^2}\right) - G(it) \left(\frac{3}{2\pi^2 t^4} - \frac{10}{3 \pi^3 t^3}\right) \right] > 0.
\]
\end{small}
We have $\widetilde{\cL}_{1, 0}(\frac{3\pi i}{10}) > 0$, since Proposition \ref{prop:d24decrease} gives $\cL_{1, 0}(\frac{3 \pi i}{10}) > 0$ and when $t = \frac{3 \pi}{10}$,
\[
    (\partial_{12}G)(it) \left(\frac{1}{\pi t^3} - \frac{10}{3 \pi^2 t^2}\right) - G(it) \left(\frac{3}{2\pi^2 t^4} - \frac{10}{3 \pi^3 t^3}\right) = -G\left(\frac{3\pi i}{10}\right) \cdot \left(\frac{5000}{81 \pi^6}\right) < 0.
\]
From Proposition \ref{prop:serrepos} (see also Remark \ref{rem:serrepos}), it is enough to show that its Serre derivative
\[
    \partial_{30}\widetilde{\cL}_{1, 0}(it) = \widetilde{\cL}_{1, 0}'(it) - \frac{5}{2} E_2(it) \widetilde{\cL}_{1, 0}(it) = - \frac{1}{2\pi} \frac{\dd \widetilde{\cL}_{1, 0}(it)}{\dd t} - \frac{5}{2} E_2(it) \widetilde{\cL}_{1, 0}(it)
\]
is positive (i.e. $t \mapsto \widetilde{\cL}_{1,0}(it) / \eta(it)^{60}$ is a monotone decreasing function in $t$) on $0 < t < \frac{3\pi}{10}$.
Recall $\partial_{30}\cL_{1, 0} = c \Delta X_{8, 2} G$ \eqref{eqn:d24L20}.
Using \eqref{eqn:serre_prod}, $\partial_{12} \Delta = 0$, and \eqref{eqn:d24ssg}, one can check that the Serre derivative of the second term of $\widetilde{\cL}_{1, 0}$ is $725760 \Delta$ times
\begin{align}
    &\partial_{18} \left[(\partial_{12}G)(it) \left(\frac{1}{\pi t^3} - \frac{10}{3 \pi^2 t^2}\right) - G(it) \left(\frac{3}{2\pi^2 t^4} - \frac{10}{3 \pi^3 t^3}\right) \right] \nonumber \\
    &= \left[\frac{37 E_4(it) - E_2(it)^2}{24}\left(\frac{1}{\pi t^3} - \frac{10}{3 \pi^2 t^2}\right) + E_2(it)\left(\frac{3}{4 \pi^2 t^4} - \frac{5}{3 \pi^3 t^3}\right) - \left(\frac{3}{\pi^3 t^5} - \frac{5}{\pi^4 t^4}\right)\right] G(it), \label{eqn:S18}
\end{align}
so $\Delta G$ factors out from $\partial_{30}\widetilde{\cL}_{1,0}(it) > 0$ and it reduces to the positivity of
\begin{equation}
    \label{eqn:d24ineq3fac}
    7560 X_{8, 2}(it) - 
    \frac{37 E_4(it) - E_2(it)^2}{24} \left(\frac{1}{\pi t^3} - \frac{10}{3 \pi^2 t^2}\right) - E_2(it)\left(\frac{3}{4 \pi^2 t^4} - \frac{5}{3 \pi^3 t^3}\right) + \left(\frac{3}{\pi^3 t^5} - \frac{5}{\pi^4 t^4}\right).
\end{equation}
Let $h(t)$ be the above function in \eqref{eqn:d24ineq3fac}.
We have (see Table \ref{tab:extform})
\[
    7560 X_{8, 2} = \frac{-7 E_2^2 E_4 + 2 E_2 E_6 + 5 E_4^2}{48}
\]
and similarly
\[
    7560 X_{8, 2}|_S = 7560 \left(X_{8, 2} + \frac{7 E_2 E_4 - E_6}{30240 \pi t} - \frac{E_4}{1440 \pi^2 t^2}\right) = \frac{-7 E_2^2 E_4 + 2 E_2 E_6 + 5 E_4^2}{48} + \frac{7 E_2 E_4 - E_6}{4 \pi t} - \frac{21 E_4}{4\pi^2 t^2}.
\]
Then
\begin{align}
    &t^{-8}h\left(\frac{1}{t}\right) \nonumber \\
    &= 7560 \left(X_{8, 2}(it) + \frac{7 E_2(it) E_4(it) - E_6(it)}{30240 \pi t} - \frac{E_4(it)}{1440 \pi^2 t^2}\right) \nonumber \\
    &\quad - \frac{1}{24} \left(37 E_4(it) - E_2(it)^2 + \frac{12 E_2(it)}{\pi t} - \frac{36}{\pi^2 t^2}\right) \left(\frac{1}{\pi t} - \frac{10}{3 \pi^2 t^2}\right) \nonumber \\
    &\quad - \left(- E_2(it) + \frac{6}{\pi t}\right) \left(\frac{3}{4 \pi^2 t^2} - \frac{5}{3 \pi^3 t^3}\right) + \left(\frac{3}{\pi^3 t^3} - \frac{5}{\pi^4 t^4}\right) \nonumber \\
    &= 7560 X_{8, 2}(it) \nonumber \\
    &\quad + \frac{1}{\pi t} \left(\frac{7 E_2(it) E_4(it) - E_6(it)}{4} - \frac{37 E_4(it) - E_2(it)^2}{24}\right) + \frac{1}{\pi^2 t^2} \left(- \frac{4 E_4(it) + 5 E_2(it)^2}{36} + \frac{E_2(it)}{4}\right) \label{eqn:htinv}
\end{align}
and since $X_{8, 2}$ is (completely) positive (Proposition \ref{prop:d2extposexceptional}), it is enough to show that \eqref{eqn:htinv} is positive, i.e. show the following inhomogeneous inequality (after factoring out $1 /\pi t$)
\begin{equation}
    \label{eqn:d24hardernonhomo}
    \frac{7 E_2(it) E_4(it) - E_6(it)}{4} - \frac{37 E_4(it) - E_2(it)^2}{24} - \frac{1}{\pi t} \left(\frac{5 E_2(it)^{2} + 4 E_{4}(it)}{36} - \frac{1}{4} E_{2}(it)\right) > 0
\end{equation}
for $t \geq \frac{10}{3 \pi}$.
We can further reduce it to an inequality with only quasimodular terms (i.e. no rational terms) with the following lemma.
\begin{lemma}
\label{lem:auxineq}
The following (nonhomogeneous) quasimodular forms are completely positive:
\begin{align}
    J_{1} &= \frac{5}{36} E_{2}^{2} + \frac{1}{9} E_{4} - \frac{1}{4} E_{2}, \\
    J_{2} &= E_{2} - E_{6}.
\end{align}
\end{lemma}
\begin{proof}
By using the Fourier expansions of $E_{2}$, $E_{4}$, $E_{6}$, and \eqref{eqn:e2der}, we can compute the Fourier expansions of the above forms explicitly as
\begin{align*}
    J_{1} &= \frac{5}{3} E_{2}' -\frac{1}{4} E_{2} + \frac{1}{4} E_{4} = \sum_{n \geq 1} (60 \sigma_{3}(n) - 40 n\sigma_{1}(n) + 6 \sigma_{1}(n)) q^{n}, \\
    J_{2} &= \sum_{n \geq 1} (504 \sigma_{5}(n) - 24 \sigma_{1}(n)) q^{n}.
\end{align*}
The complete positivity of $J_{1}$ follows from the trivial estimates $\sigma_{3}(n) \ge n^{3}$ and $\sigma_{1}(n) \leq 1 + 2 + \cdots + n = \frac{n(n+1)}{2} \leq n^2$, and that of $J_{2}$ follows from $504 \sigma_{5}(n) - 24 \sigma_{1}(n) = \sum_{d|n} (504 d^{5} - 24 d) > 0$.
\end{proof}
By Lemma \ref{lem:auxineq}, for $t \geq \frac{10}{3 \pi}$ we have
\begin{align}
    &\frac{7 E_2(it) E_4(it) - E_6(it)}{4} - \frac{37 E_4(it) - E_2(it)^2}{24}  - \frac{1}{\pi t} \left(\frac{5E_{2}(it)^{2} + 4 E_{4}(it)}{36}  - \frac{E_{2}(it)}{4}\right) \nonumber \\
    &\ge \frac{7 E_2(it) E_4(it) - E_6(it)}{4} - \frac{37 E_4(it) - E_2(it)^2}{24} - \frac{3}{10} \left(\frac{5E_{2}(it)^{2} + 4 E_{4}(it)}{36}  - \frac{E_{2}(it)}{4}\right) \label{eqn:usej1pos} \\
    &= \frac{7 E_2(it) E_4(it) - E_6(it)}{4} - \frac{63}{40} E_{4}(it) + \frac{3}{40} E_{2}(it) \nonumber \\
    &> \frac{7 E_2(it) E_4(it) - E_6(it)}{4} - \frac{63}{40} E_{4}(it) + \frac{3}{40} E_{6}(it) \label{eqn:usej2pos} \\
    &= \frac{7}{4} \left(E_{2}(it) E_{4}(it) - \frac{1}{10}E_{6}(it) - \frac{9}{10}E_{4}(it)\right) =: \frac{7}{4} J_{3} \label{eqn:d24hardernonhomo2}
\end{align}
where the positivity of $J_{1}$ and $J_{2}$ are used in \eqref{eqn:usej1pos} and \eqref{eqn:usej2pos}, respectively.
Now, we can prove the positivity of \eqref{eqn:d24hardernonhomo2} (i.e. $J_{3}$) as follows.
As in Lemma \ref{lem:auxineq}, we can compute the Fourier expansion of $J_{3}$ as
\begin{align*}
    J_{3} &=E_{2} E_{4} - \frac{1}{10} E_{6} - \frac{9}{10} E_{4} \\
    &= 3 E_{4}' + \frac{9}{10} E_{6} - \frac{9}{10} E_{4} \\
    &= \sum_{n \geq 1} \left(720 n \sigma_{3}(n)- \frac{2268}{5} \sigma_{5}(n) - 216 \sigma_{3}(n)\right) q^{n}  \\
    &=: \sum_{n \geq 1} a_{n}q^{n}.
\end{align*}
We have $a_{1} = \frac{252}{5} > 0$.
For $n \geq 2$,
\[
n \sigma_{3}(n) \leq n (1^{3} + 2^{3} + \cdots + n^{3}) = \frac{n^{3}(n+1)^{2}}{4} \leq \frac{9}{16} n^{5} < \frac{9}{16} \sigma_{5}(n) < \frac{2268}{720 \cdot 5} \sigma_{5}(n)
\]
and we get $a_{n} < 0$.
From this observation, the function 
\[
    t \mapsto e^{2\pi t} J_{3}(it) = a_{1} + \sum_{n\geq 2}a_{n}e^{-2\pi(n-1)t}
\]
is monotone increasing, and by \eqref{eqn:eisval}
\[
    e^{2\pi t} J_{3}(it) \geq e^{2 \pi} J_{3}(i) = e^{2\pi} \left(\frac{3}{\pi} - \frac{9}{10}\right) E_{4}(i) > 0 \Rightarrow J_{3}(it) > 0
\]
for $t \geq 1$, hence for $t > \frac{10}{3 \pi}$.
\end{proof}

\begin{remark}
The estimate $\pi < \frac{10}{3}$ is used in the proof (e.g. \eqref{eqn:usej1pos}), and this can be verified \emph{geometrically} (without calculators) by considering the area of a regular octagon circumscribed to a unit circle:
\[
    \pi < 8 \tan\left(\frac{\pi}{8}\right) = 8 (\sqrt{2} - 1) < \frac{10}{3}.
\]
\end{remark}

\begin{remark}
The inequalities \eqref{eqn:d24ineq3new2} and \eqref{eqn:d24ineq3fac} are ``homogeneous'' if one regards $\frac{1}{t} = \frac{i}{z}$ and $\frac{1}{\pi}$ as ``weight 1'' objects, which makes sense if we consider the transformation law of $E_2$ \eqref{eqn:e2trans}.
However, we had to flip it under $t \leftrightarrow \frac{1}{t}$ and prove the nonhomogeneous inequalities \eqref{eqn:d24hardernonhomo}-\eqref{eqn:d24hardernonhomo2} instead.
It would be interesting if one can prove the inequality \eqref{eqn:d24ineq3new} in a purely homogeneous way.
\end{remark}


\bibliographystyle{plain} 
\bibliography{refs} 

\newpage
\section*{Appendix}
\label{sec:appendix}
\begin{appendices}

\section{Brief description of SageMath codes}
\label{subsec:appendix_sage}

All the codes are available in the GitHub repository \url{https://github.com/seewoo5/posqmf}.
Our codes are based on the current implementation of quasimodular forms, thanks to David Ayotte.
In SageMath, the ring of quasimodular forms of level $\Gamma_0(N)$ or $\Gamma_1(N)$ are essentially implemented as polynomial rings of one variable ($E_2$) with the ring of (genuine) modular forms as a coefficient ring, based on \cite{kaneko1998supersingular}.
For the quasimodular forms of level $\SL_2(\mathbb{Z})$, we can simply define the ring as \texttt{QM = QuasiModularForms(1)}.
However, for the ring of quasimodular forms of level $\Gamma(2)$, we had to implement it ourselves from scratch since the current implementation of the modular forms does not support the level.
Especially, they can have $q$-expansion with half-integral powers (e.g. \eqref{eqn:h2}), where SageMath does not support power series with fractional powers.
The ring of quasimodular forms of level $\Gamma(2)$ is isomorphic to a polynomial ring with three generators, namely $H_2 = \Theta_2^4$, $H_4 = \Theta_4^4$, and $E_2$.
So we simply define it as a polynomial ring \texttt{QM2.<H2,H4,E2\_> = QQ[`H2,H4,E2']}.
Also, for the proof of the harder inequality (Section \ref{subsec:d24harder}), we implement auxiliary rings \texttt{RQM} and \texttt{RQM2}, which correspond to 
$$
    \RQM(\Gamma) := \QM(\Gamma)\left[\frac{1}{\pi}, \frac{i}{z}\right]
$$
for $\Gamma = \SL_2(\bZ)$ and $\Gamma = \Gamma(2)$.
We also implemented (Serre-)derivatives and $S$-actions on these rings, which are used to check the intermediate steps of the computations of the limits (Proposition \ref{prop:d8limit} and \ref{prop:d24limit}) and the intermediate steps of the proof of Proposition \ref{prop:d24harder}.

The repository consists of a directory \texttt{posqmf} and a Jupyter notebook \href{https://github.com/seewoo5/posqmf/blob/main/sphere_packing_ineq.ipynb}{\texttt{sphere\_packing\_ineq.ipynb}}.
The utility files are located under the directory \href{https://github.com/seewoo5/posqmf/tree/main/posqmf/sage}{\texttt{posqmf/sage}}, including basic APIs for computing quasimodular forms.
There are four \texttt{.sage} files in total:
\begin{itemize}
    \item \texttt{extremal\_qm.sage}: It computes an extremal quasimodular form of given weight and depth. For depth 1 and 2, it uses the recurrence relations in \cite{grabner2020quasimodular} mentioned in Section \ref{sec:extremal}.
    \item \texttt{utils\_l1.sage}: It contains basic APIs for level 1 quasimodular forms, including the computation of $q$-expansions and (Serre-)derivatives.
    \item \texttt{utils\_l2.sage}: It contains basic APIs for level $\Gamma(2)$ quasimodular forms. It also provides a function \texttt{l1\_to\_l2} that converts a level 1 quasimodular form into a level $\Gamma(2)$ quasimodular form.
    \item \texttt{utils\_rqm.sage}: It contains basic APIs for the rings \texttt{RQM} and \texttt{RQM2}. The (Serre-)derivative on $\QM(\Gamma)$ is extended to $\RQM(\Gamma)$, and slash action under $S$ for homogeneous elements are implemented.
\end{itemize}
See \href{https://github.com/seewoo5/posqmf/blob/main/posqmf/README.md}{\texttt{posqmf/README.md}} for more details.

The Jupyter notebook \href{https://github.com/seewoo5/posqmf/blob/main/sphere_packing_ineq.ipynb}{\texttt{sphere\_packing\_ineq.ipynb}} consists of three parts.
The first part titled \textbf{Extremal quasimodular forms} checks the identities \eqref{eqn:dx82}-\eqref{eqn:dx142} of Proposition \ref{prop:d2extposexceptional}.
Second part \textbf{Dimension 8} checks the identities \eqref{eqn:d8ssf} and \eqref{eqn:d8ssg}, along with the identity \eqref{eqn:d8ineqfactor} used in the alternative proof of the monotonicity (Remark \ref{rem:d8proof2}) and the expansion of $F|S$ and $G|S$ for the limit \eqref{eqn:d8lim}.
The last part \textbf{Dimension 24} verifies the identities \eqref{eqn:serrederF16}, \eqref{eqn:d24ssf}, \eqref{eqn:d24ssg}, the expansion of $F|S$ and $G|S$ for the limit \eqref{eqn:d24lim}, along with some of the intermediate steps in Section \ref{subsec:d24harder}.
It takes a few seconds to run the whole notebook.

\newpage

\section{Pen-and-paper proofs of the quasimodular form identities}
\label{subsec:penpaperproof}

In this section, we give pen-and-paper proofs of the quasimodular form identities.

\subsection{Useful identities}
\label{subsec:appendix_identities}

From \eqref{eqn:serre_prod} and \eqref{eqn:serre_e2}-\eqref{eqn:serre_e6}, we have
\begin{align}
    \partial_{w}(E_2F) &= E_2 (\partial_{w-1}F) - \frac{1}{12} E_4 F, \label{eqn:serre_e2}\\
    \partial_{w}(E_4F) &= E_4 (\partial_{w-4}F) - \frac{1}{3}E_6 F, \label{eqn:serre_e4} \\
    \partial_{w}(E_6F) &= E_6 (\partial_{w-6}F) - \frac{1}{2} E_4^2 F, \label{eqn:serre_e6}
\end{align}
and
\begin{align}
    \partial_{12} \Delta &= \frac{\partial_{12}(E_4^3 - E_6^2)}{1728} = \frac{3 E_4^2 \cdot (-\frac{1}{3} E_6) - 2 E_6 \cdot (-\frac{1}{2} E_4^2)}{1728} = 0 \label{eqn:serredisc}\\
    \partial_{w} (\Delta F) &= (\partial_{12}\Delta) F + \Delta (\partial_{w-12}F) = \Delta (\partial_{w-12}F).\label{eqn:serredisc2}
\end{align}
Also, combining the product rule with the Serre derivatives of theta functions, we have the following formula:
For $a, b \geq 0$, we have
\begin{equation}
\label{eqn:thetaprodserreder}
\partial_{2a + 2b} (H_{2}^{a} H_{4}^{b}) = \frac{1}{6} H_{2}^{a}H_{4}^{b}((a-2b)H_{2} + (2a-b)H_{4}).
\end{equation}

The general ideas to prove the identities like \eqref{eqn:dx82} or \eqref{eqn:d24ssf} by hand are two ways, (1) by matching the first few $q$-coefficients and invoking the uniqueness of extremal quasimodular forms \cite[Theorem 1.3, p. 403]{pellarin2020extremal}, or (2) by expressing them in terms of Eisenstein series or theta functions and use Ramanujan's identities \eqref{eqn:e2der}-\eqref{eqn:e4der} or \eqref{eqn:thetaprodserreder}.

\subsection{Proof of the depth 2 extremal form identities \eqref{eqn:dx82}-\eqref{eqn:dx142}}
\label{subsec:appendix_xd2}

Using Lemma \ref{lem:lowpos} and the $q$-expansions in Table \ref{tab:extform}, one can check that the first dimension-many $q$-coefficients are matching, and \cite[Theorem 1.3, p. 403]{pellarin2020extremal} implies the identity.
For example, both $X_{8, 2}'$ and $2 X_{4, 2} X_{6, 1}$ have weight 10 and depth 3, where $\dim \QM_{10, 3} = 4$ \cite[eq. (2.10), p. 2237]{grabner2020quasimodular}.
Their $q$-expansions are
\begin{align*}
    X_{8, 2}' &= 2q^2 + 48q^3 + O(q^4) \\
    2 X_{4, 2} X_{6, 1} &= 2(q + 6q^2 + O(q^3))(q + 18 q^2 + O(q^3)) = 2q^2 + 48q^3 + O(q^4)
\end{align*}
hence $X_{8, 2}' - 2 X_{4, 2} X_{6, 1} = O(q^4)$.
By the uniqueness of extremal quasimodular forms \cite[Theorem 1.3]{pellarin2020extremal}, this has to be a constant multiple of $X_{10, 3} = q^3 + O(q^4)$, so is identically zero. This proves \eqref{eqn:dx82}, and the other identities can be proved similarly.

\subsection{Proof of the 8-dimensional identities \eqref{eqn:d8ssf}, \eqref{eqn:d8ssg}}
\label{subsec:appendix_d8}

The following lemma will be useful in the proofs.
\begin{lemma}
    \label{lem:auxidentity}
    We have
    \begin{equation}
        \label{eqn:auxidentity}
        \frac{49}{24} X_{8, 1}^2 - \frac{25}{24} E_4 X_{6, 1}^2 = \Delta X_{4, 2}.
    \end{equation}
\end{lemma}
\begin{proof}
    By Lemma \ref{lem:lowpos},
    \begin{align*}
        49 X_{8, 1}^2 - 25 E_4 X_{6, 1}^2 &= \frac{7^2}{12^4 \cdot 7^2}  (E_4^2 - E_2 E_6)^2 - \frac{5^2}{12^4 \cdot 5^2} E_4 (E_2 E_4 - E_6)^2 \\
        &= \frac{(E_4^3 - E_6^2)(E_4 - E_2^2)}{12^4} \\
        &= 2 \cdot 12 \cdot \Delta X_{4, 2}
    \end{align*}
    and we get the desired identity.
\end{proof}

To prove \eqref{eqn:d8ssf}, note that $F = 720^2 X_{6, 1}^2$. Since $\frac{12^3 \cdot 100}{720^2}  = \frac{1}{3}$, \eqref{eqn:d8ssf} is equivalent to
\begin{equation}
\label{eqn:d8ssf2}
\partial_{10}^2 X_{6, 1}^2 = \frac{5}{6} E_4 X_{6, 1}^2 + \frac{1}{3} \Delta X_{4, 2}.
\end{equation}
By \eqref{eqn:serre_prod}, \eqref{eqn:d1rec1}${}_{6}$ and \eqref{eqn:oded1serre}${}_{6}$, we have
\begin{align*}
    \partial_{10} X_{6, 1}^2 &= 2 X_{6, 1} \partial_{5} X_{6, 1}, \\
    \partial_{10}^2 X_{6, 1}^2 &= 2 ((\partial_{5} X_{6, 1})^2 + X_{6, 1} \partial_{5}^2 X_{6, 1}) = \frac{7}{72} (7 X_{8, 1}^2 + 5 E_4 X_{6, 1}^2),
\end{align*}
so Lemma \ref{lem:auxidentity} implies
\[
    \partial_{10}^2 X_{6, 1}^2 - \frac{5}{6} E_4 X_{6, 1}^2 = \frac{49 X_{8, 1}^2 - 25 E_4 X_{6, 1}^2}{72} = \frac{1}{3} \Delta X_{4,2}
\]
and we get \eqref{eqn:d8ssf2}.
For \eqref{eqn:d8ssg}, use \eqref{eqn:thetaprodserreder} twice to compute $\partial_{10}^2 G$, then rewriting $E_4$ and $\Delta$ using \eqref{eqn:e4theta} and \eqref{eqn:disctheta} completes a proof. \qed

\subsection{Alternative proof of Proposition \ref{prop:d8decrease}}
\label{subsec:appendix_d8newproof}
We give another proof of Proposition \ref{prop:d8decrease} as mentioned in Remark \ref{rem:d8proof2}.
\begin{lemma}
We have    
\begin{equation}
\label{eqn:d8ineqfactor}
    \cL_{1, 0} = \frac{15}{2} H_2^3 (H_2 + H_4)^2 H_4^2 (E_2 E_4 - E_6)\left(E_4 - \frac{1}{2}E_2(H_2 + 2H_4)\right).
\end{equation}
\end{lemma}
\begin{proof}
By using \eqref{eqn:serre_prod}, \eqref{eqn:thetaprodserreder}, \eqref{eqn:e4theta} and Ramanujan's identities,
\begin{align*}
    \partial_{10}F &= \frac{5}{6} (E_2 E_4 - E_6) (E_4^2 - E_2 E_6) \\
    \partial_{10}G &= \frac{10}{6} H_2^3 (H_2^3 + 3 H_2^2 H_4 + 3 H_2 H_4^2 + 2 H_4^3)  \\
    &= \frac{10}{6} H_2^3 (H_2 + 2 H_4) (H_2^2 + H_2 H_4 + H_4^2) = \frac{5}{3} H_2^3 (H_2 + 2H_4) E_4
\end{align*}
Substituting $\partial_{10}F$ and $\partial_{10}G$ gives
\begin{align*}
    &(\partial_{10} F) G - F (\partial_{10}G) \\
    &= \frac{5}{6}H_2^3 (E_2 E_4 - E_6) \left[(E_4^2 - E_2 E_6) (2H_2^2 + 5 H_2 H_4 + 5 H_4^2) - 2 (E_2 E_4 - E_6) E_4 (H_2 + 2 H_4)\right].
\end{align*}
We claim that the third factor is
\begin{equation}
\label{eqn:L10factor}
9 H_4^2 (H_2 + H_4)^2 \left(E_4 - \frac{1}{2} E_2 (H_2 + 2H_4)\right).
\end{equation}
To prove this, we regroup the factor according to depth, and write it as a form of $A + B E_2$ where
\begin{align*}
    A &= E_4^2 (2H_2^2 + 5 H_2 H_4 + 5 H_4^2) + 2 E_4 E_6 (H_2 + 2 H_4) \\
    B &= - E_6 (2H_2^2 + 5 H_2 H_4 + 5 H_4^2) - 2 E_4^2 (H_2 + 2 H_4).
\end{align*}
Let $J := H_2 + 2 H_4$. We will use the following identities for simpler computations: from \eqref{eqn:e4theta} and \eqref{eqn:e6theta},
\begin{align*}
    2 H_2^2 + 5 H_2 H_4 + 5 H_4^2 &= H_2^2 + H_2 H_4 + H_4^2 + (H_2 + 2H_4)^2 = E_4 + J^2, \\
    E_4 - J^2 &= H_2^2 + H_2 H_4 + H_4^2 - (H_2 + 2H_4)^2 =-3 H_4 (H_2 + H_4) \\
    E_6 
    &= \frac{1}{2} J (-2H_2^2 + H_2 H_4 + H_4^2) = \frac{1}{2} J(-2E_4 + 3 H_4(H_2 + H_4))
\end{align*}
Now, $A$ is
\begin{align*}
    E_4^2 (2 H_2^2 + 5 H_2 H_4 + 5 H_4^2) + 2 E_4 E_6 (H_2 + 2 H_4)
    &= E_4^2 (E_4 + J^2) + J^2 E_4 (-2E_4 + 3 H_4 (H_2 + H_4)) \\
    &= E_4 (E_4 (E_4 - J^2) + 3 J^2 H_4 (H_2 + H_4)) \\
    &= 3 E_4 H_4 (H_2 + H_4) (-E_4 + J^2) = 9 E_4 H_4^2 (H_2 + H_4)^2
\end{align*}
and $B$ is
\begin{align*}
    - E_6 (2H_2^2 + 5 H_2 H_4 + 5 H_4^2) - 2 E_4^2 (H_2 + 2 H_4) &= \frac{1}{2}J(2E_4 - 3 H_4 (H_2 + H_4)) (E_4 + J^2) - 2 E_4^2 J \\
    &= \frac{J}{2} (-2 E_4^2 + 2 E_4 J^2 - 3 E_4 H_4 (H_2 + H_4) - 3 J^2 H_4 (H_2 + H_4)) \\
    &= \frac{J}{2} (6 E_4 H_4 (H_2 + H_4) - 3 E_4 H_4 (H_2 + H_4) - 3 J^2 H_4 (H_2 + H_4)) \\
    &= \frac{3J}{2} H_4 (H_2 + H_4) (E_4 - J^2) \\
    &= \frac{9J}{2}H_4^2 (H_2 + H_4)^2
\end{align*}
and the identity follows.
\end{proof}
Notice that the first four factors of \eqref{eqn:d8ineqfactor} are all positive, where the fourth factor is $E_2 E_4 - E_6 = 3 E_4' = 720 X_{6, 1}$.
Hence the positivity of \eqref{eqn:d8ineqfactor} reduces to that of the last factor.
\begin{lemma}
$L := E_4 - \frac{1}{2} E_2 (H_2 + 2 H_4)$ is positive.
\end{lemma}
\begin{proof}
With \eqref{eqn:e4theta}, it can be decomposed as
\[
E_4 - \frac{1}{2} E_2 (H_2 + 2H_4) = \frac{3}{4} H_2^2 + \frac{1}{4} (H_2 + 2 H_4) (H_2 + 2 H_4 - 2 E_2)
\]
and the second term of the last factor is positive, since (from \eqref{eqn:h2} and \eqref{eqn:h4})
\[
H_2 + 2 H_4 - 2 E_2 = (H_2 + 2 H_4 - 2) + 2 (1 - E_2) = 2 \sum_{n \ge 1} r_4(2n) q^n + 48 \sum_{n \ge 1} \sigma_1(n) q^n.
\]
\end{proof}

\subsection{Proof of the 24-dimensional identities \eqref{eqn:serrederF16}, \eqref{eqn:d24ssf}, \eqref{eqn:d24ssg}}
\label{subsec:appendix_d24}

The following lemma will be useful.
\begin{lemma}
    \label{lem:F16identity}
    We have
    \begin{equation}
    \label{eqn:F16eq1}
        F = 2 \cdot 7^2 \cdot 12^5 (X_{8, 1}^2 - \Delta X_{4, 2}).
    \end{equation}
\end{lemma}
\begin{proof}
    We can rewrite $F$ as
    \begin{align*}
        F &= 49 E_2^2 E_4^3 - 25 E_2^2 E_6^2 - 48 E_2 E_4^2 E_6 - 25 E_4^4 + 49 E_4 E_6^2 \\
        &= 49 (E_2^2 - E_4) (E_4^3 - E_6^2) + 24 (E_4^2 - E_2 E_6)^2
    \end{align*}
    Now apply Lemma \ref{lem:lowpos} and $E_4^3 - E_6^2 = 12^3 \Delta$ to conclude that
    \[
        F = 7^2 \cdot (-2 \cdot 12^2 X_{4, 2}) (12^3 \Delta) + 2 \cdot 12 \cdot 7^2 \cdot 12^4 X_{8, 1}^2 = 2 \cdot 7^2 \cdot 12^5 (X_{8, 1}^2 - \Delta X_{4, 2}).
    \]
\end{proof}

\begin{proof}[Proof of \eqref{eqn:serrederF16}]
From $6706022400 = 7^2 12^5 \cdot 2 \cdot 5^2 \cdot 11$ and Lemma \ref{lem:F16identity}, \eqref{eqn:serrederF16} is equivalent to
\begin{equation}
    \label{eqn:serrederF16_2}
    \partial_{14}(X_{8, 1}^2 - \Delta X_{4, 2}) = 5^2 \cdot 11 X_{6, 1} X_{12, 1}.
\end{equation}
Since both of sides of \eqref{eqn:serrederF16_2} have weight 18 and depth $\le 2$ (note that $14 = 16 - 2$, so the left hand side has depth $\le 2$ by \cite[Proposition 3.3, p. 467]{kaneko2006extremal}), $\dim \QM_{18, 2} = 5$ \cite[eq. (2.10), p. 2237]{grabner2020quasimodular} and \cite[Theorem 1.3, p. 403]{pellarin2020extremal} imply that we only need to verify that the first 5 $q$-coefficients matches.
From \eqref{eqn:F16eq1} and Lemma \ref{lem:lowpos},
\begin{align}
    &X_{8, 1}^2 - X_{4, 2} \Delta \nonumber \\
    &= \left(\sum_{n \ge 1} n \sigma_5(n) q^n\right)^2 - \left(\sum_{n \ge 1} n \sigma_1(n) q^n\right) q \prod_{n \ge 1} (1 - q^n)^{24} \nonumber\\
    &=  (q + 66 q^2 + 732 q^3 + O(q^4))^2 - (q + 6 q^2 + 12 q^3 + O(q^4))(q - 24 q^2 + 252 q^3 + O(q^4)) \nonumber \\
    &= 150 (q^3 + 38 q^4 + O(q^5)) \label{eqn:F16qexp}
\end{align}
and by \eqref{eqn:serrederqexp}, we have
\begin{align*}
    &\partial_{14} (X_{8, 1}^2 - X_{4, 2} \Delta) \\
    &= 150 \left[(3q^3 +152q^4 + O(q^5)) - \frac{7}{6} (1 - 24q - 72q^2 + O(q^3))(q^3 + 38q^4 + O(q^5))\right] \\
    &= 275 q^3 + 20350 q^4 + O(q^5).
\end{align*}
For the right hand side of \eqref{eqn:serrederF16_2}, \eqref{eqn:d1rec3} gives
\begin{equation}
    \label{eqn:x121}
    X_{12, 1} = \frac{12}{864 \cdot 11} (E_4 X_{8, 1} - E_6 X_{6, 1}) = \frac{1}{72 \cdot 11} (E_4 X_{8, 1} - E_6 X_{6, 1})
\end{equation}
and by Lemma \ref{lem:lowpos},
\begin{align*}
    5^2 \cdot 11 X_{6, 1} X_{12, 1} &= \frac{5^2 \cdot 11}{72 \cdot 11} \left(\sum_{n \ge 1} n\sigma_3(n)q^n\right) \left[E_4 \left(\sum_{n \ge 1} n \sigma_5(n) q^n\right) - E_6 \left(\sum_{n \ge 1} n \sigma_3(n) q^n\right)\right] \\
    &= \frac{5^2}{72} (q + 18 q^2 + O(q^3)) \times [(1 + 240q + 2160 q^2 + O(q^3))(q + 66 q^2 + 732 q^3 + O(q^4)) \\
    &\quad - (1 - 504q - 16632 q^2 + O(q^3))(q + 18q^2 + 84q^3 + O(q^4))] \\
    &= 275 q^3 + 20350 q^4 + O(q^5).
\end{align*}
This shows $\partial_{14}(X_{8, 1}^2 - X_{4, 2}\Delta) - 5^2 \cdot 11 X_{6, 1} X_{12, 1} = O(q^5)$.
The difference has to be a constant multiple of $X_{18, 2} = q^4 + O(q^5)$, hence zero.
Note that \eqref{eqn:F16qexp} also shows that the first nonzero $q$-coefficient of $F$ is positive, hence $F(it) > 0$ for sufficiently large $t > 0$.
\end{proof}

\begin{proof}[Proof of \eqref{eqn:d24ssf} and \eqref{eqn:d24ssg}]
To show \eqref{eqn:d24ssf}, Lemma \ref{lem:F16identity} and $\frac{2^{11} \cdot 3^7 \cdot 5^2 \cdot 7^2}{2 \cdot 7^2 \cdot 12^5} = 225$ shows that \eqref{eqn:d24ssf} is equivalent to
\[
\partial_{14}^2 (X_{8, 1}^2 - \Delta X_{4, 2}) = \frac{14}{9} E_4 (X_{8, 1}^2 - \Delta X_{4, 2}) + 225 \Delta X_{8, 2}.
\]
By the product rule \eqref{eqn:serre_prod} and \eqref{eqn:serredisc2},
\begin{align*}
    \partial_{14}(X_{8, 1}^2 - \Delta X_{4, 2}) &= 2 X_{8, 1} \cdot \partial_{7}X_{8, 1} - \Delta \partial_{2} X_{4, 2} = \frac{5}{6} E_4 X_{6,1} X_{8, 1} - \Delta \partial_2 X_{4, 2} \\
    \partial_{14}^{2}(X_{8, 1}^2 - \Delta X_{4, 2}) &= -\frac{5}{18} E_6 X_{6, 1} X_{8, 1} + \frac{35}{72} E_4 X_{8, 1}^2 + \frac{25}{72} E_4^2 X_{6, 1}^2 - \Delta\partial_{2}^2 X_{4, 2} \\
    &= -\frac{5}{18} E_6 X_{6, 1} X_{8, 1} + \frac{35}{72} E_4 X_{8, 1}^2 + \frac{25}{72} E_4^2 X_{6, 1}^2 -\Delta  \left(\frac{5}{9} E_4 X_{4, 2} - 105 X_{8, 2}\right)
\end{align*}
where \eqref{eqn:d1rec1}${}_{6}$, \eqref{eqn:d1rec4}${}_{6}$, \eqref{eqn:d2eq1}${}_{4}$ are used in the last two equalities. By Lemma \ref{lem:auxidentity}, 
\begin{align*}
    \partial_{14}^2 (X_{8, 1}^2 - \Delta X_{4, 2}) - \frac{14}{9} E_4 (X_{8, 1}^2 - \Delta X_{4, 2}) &= \frac{35}{36} E_4 X_{8, 1}^2 - \frac{5}{18} E_6 X_{6, 1} X_{8, 1} -\frac{25}{36} E_4^2 X_{6, 1}^2 + 105 \Delta X_{8, 2}
\end{align*}
so it is enough to prove
\[
\frac{35}{36} E_4 X_{8, 1}^2 - \frac{5}{18} E_6 X_{6, 1} X_{8, 1} -\frac{25}{36} E_4^2 X_{6, 1}^2 = \frac{5}{36} (7 E_4 X_{8, 1}^2 - 2 E_6 X_{6, 1} X_{8,1 } - 5 E_4^2 X_{6, 1}^2) = 120 \Delta X_{8, 2}.
\]
From Lemma \ref{lem:lowpos} and $720 = 5 \cdot 12^2, 1008 = 7 \cdot 12^2$,
\begin{align*}
    &7 E_4 X_{8, 1}^2 - 2 E_6 X_{6, 1} X_{8, 1} - 5 E_4^2 X_{6, 1}^2 &\\
    &= \frac{1}{5 \cdot 7 \cdot 12^4} (5 E_4(- E_2 E_6 + E_4^2)^2 - 2 E_6 (E_2 E_4 - E_6) (-E_2 E_6 + E_4^2) - 7 E_4^2 (E_2 E_4 - E_6)^2) \\
    &= \frac{1}{5 \cdot 7 \cdot 12^4} (-7E_2^2 E_4 + 2 E_2 E_6 + 5 E_4^2) (E_4^3 - E_6^2) \\
    &= \Delta \cdot \frac{-7E_2^2 E_4 + 2 E_2 E_6 + 5 E_4^2}{5 \cdot 7 \cdot 12} \\
    &= \Delta \cdot 864 X_{8, 2}
\end{align*}
(see Table \ref{tab:extform}, where $\frac{362880}{5 \cdot 7 \cdot 12} = 864$), and $\frac{5 \cdot 864}{36} = 120$ proves the desired identity.

The identity \eqref{eqn:d24ssg} can be proved similarly as \eqref{eqn:d8ssg}. By applying \eqref{eqn:thetaprodserreder} twice, we can compute $\partial_{14}^2 G$ as
\begin{align*}
\partial_{14}^2 G &= \frac{14}{9} H_2^5 (2H_2^4 + 9 H_2^3 H_4 + 16 H_2^2 H_4^2 + 14 H_2 H_4^3 + 7 H_4^4) \\
&= \frac{14}{9} (H_2^2 + H_2 H_4 + H_4^2) \cdot H_2^5(2 H_2^2 + 7 H_2 H_4 + 7 H_4^2),
\end{align*}
and the last equation is equal to $\frac{14}{9} E_4 G$ by \eqref{eqn:e4theta}.
\end{proof}

\newpage
\section{Table of extremal forms}
\label{subsec:exttab}

\renewcommand*{\arraystretch}{2}

Table \ref{tab:extform} gives first few extremal quasimodular forms of Kaneko and Koike \cite{kaneko2006extremal}.

\begin{table}[h]
\begin{center}
    \begin{tabular}{c|c|l}
        \toprule
        $s$ & $w$ & $X_{w, s}$ \\
        \midrule
        \multirow{5}{*}{$1$} & $6$ & $\dfrac{E_2 E_4 - E_6}{720} = q + 18 q^2 + 84 q^3 + 292 q^4 + 630 q^5 + \cdots$ \\
            & $8$ & $\dfrac{-E_2 E_6 + E_4^2}{1008} = q + 66q^2 + 732q^3 + 4228q^4 + 15630q^5 + \cdots$ \\
            & $10$ & $\dfrac{E_2 E_4^2 - E_4 E_6}{720} = q + 258q^2 + 6564q^3 + 66052q^4 + 390630q^5 + \cdots$ \\
            & $12$ & $\dfrac{-12 E_2 E_4 E_6 + 5E_4^3 + 7E_6^2}{3991680} = q^2 + 56q^3 + 1002q^4 + 9296q^5 + 57708q^6 + \cdots$\\
            & $14$ & $\dfrac{7 E_2 E_4^3 + 5 E_2 E_6^2 - 12 E_4^2 E_6}{4717440} = q^2 + 128 q^3 + 4050 q^4 + 58880 q^5 + 525300 q^6 + \cdots$\\
        \midrule
        \multirow{5}{*}{$2$} & $4$ & $\dfrac{-E_2^2 + E_4}{288} = q + 6q^2 + 12q^3 + 28q^4 + 30q^5 + \cdots$ \\
            & $8$ & $\dfrac{-7 E_2^2 E_4 + 2 E_2 E_6 + 5 E_4^2}{362880} = q^2 + 16 q^3 + 102 q^4 + 416 q^5 + 1308 q^6 + \cdots$\\
            & $10$ & $\dfrac{5 E_2^2 E_6 + 2 E_2 E_4^2 - 7 E_4 E_6}{1088640} = q^2 + \dfrac{104}{3}q^3 + 390q^4 + 2480q^5 + 11140 q^6 + \cdots$\\
            & $12$ & $\dfrac{-77 E_2^2 E_4^2 + 34 E_2 E_4 E_6 + 50 E_4^3 - 7 E_6^2}{798336000} = q^3 + \dfrac{51}{2}q^4 + \dfrac{1422}{5}q^5 + 1944q^6 + 9714q^7 + \cdots$\\
            & $14$ & $\dfrac{13 E_2^2 E_4 E_6 + E_2 E_4^3 - 3 E_2 E_6^2 - 11 E_4^2 E_6}{415134720} = q^3 + \dfrac{93}{2}q^4 + 810q^5 + 8004q^6 + 54474q^7 + \cdots$ \\
        \bottomrule
    \end{tabular}
    \caption{Extremal forms of depth $\leq 2$ and weight $\leq 14$.}
    \label{tab:extform}
\end{center}
\end{table}

\end{appendices}

\end{document}